\documentclass[11pt]{amsart}

\usepackage{setspace}

\usepackage{amstext}
\usepackage{amsthm}
\usepackage{amssymb}
\usepackage[a4paper, width=6in, height=9in, centering]{geometry}
\usepackage{enumerate}
\usepackage{xcolor}
\usepackage{comment}
\usepackage{hyperref}
\usepackage{graphicx}

\newtheorem{theorem}{Theorem}[section]
\newtheorem{lemma}[theorem]{Lemma}
\newtheorem{proposition}[theorem]{Proposition}
\newtheorem{corollary}[theorem]{Corollary}
\newtheorem{remark}[theorem]{Remark}

\DeclareMathOperator{\loc}{loc}

\DeclareMathOperator{\dist}{dist}

\DeclareMathOperator{\tr}{tr}

\DeclareMathOperator{\supp}{supp}

\DeclareMathOperator{\Div}{div}
\DeclareMathOperator{\dvol}{dvol}

\DeclareMathOperator{\R}{\mathbb{R}}
\DeclareMathOperator{\Z}{\mathbb{Z}}

\DeclareMathOperator{\ind}{ind}
\DeclareMathOperator{\grad}{grad}
\DeclareMathOperator{\link}{link}

\newcommand{\mres}{\mathbin{\vrule height 1.6ex depth 0pt width
0.13ex\vrule height 0.13ex depth 0pt width 1.3ex}}

\title{Plateau's problem via the Allen--Cahn functional}

\author{Marco A. M. Guaraco}

\author{Stephen Lynch}

\address{Department of Mathematics, Imperial College London, London SW7 2AZ, United Kingdom}
\email{guaraco@imperial.ac.uk}
\email{stephen.lynch@imperial.ac.uk}

\begin{document}

\maketitle

\begin{abstract}
Let $\Gamma$ be a compact codimension-two submanifold of $\mathbb{R}^n$, and let $L$ be a nontrivial real line bundle over $X = \mathbb{R}^n \setminus \Gamma$. We study the Allen--Cahn functional, \[E_\varepsilon(u) = \int_X \varepsilon \frac{|\nabla u|^2}{2} + \frac{(1-|u|^2)^2}{4\varepsilon}\,dx,\] on the space of sections $u$ of $L$. Specifically, we are interested in critical sections for this functional and their relation to minimal hypersurfaces with boundary equal to $\Gamma$. We first show that, for a family of critical sections with uniformly bounded energy, in the limit as $\varepsilon \to 0$, the  associated family of energy measures converges to an integer rectifiable $(n-1)$-varifold $V$. Moreover, $V$ is stationary with respect to any variation which leaves $\Gamma$ fixed. Away from $\Gamma$, this follows from work of Hutchinson--Tonegawa; our result extends their interior theory up to the boundary $\Gamma$. 

Under additional hypotheses, we can say more about $V$. When $V$ arises as a limit of critical sections with uniformly bounded Morse index, $\Sigma := \supp \|V\|$ is a minimal hypersurface, smooth away from $\Gamma$ and a singular set of Hausdorff dimension at most $n-8$. If the sections are globally energy minimizing and $n = 3$, then $\Sigma$ is a smooth surface with boundary, $\partial \Sigma = \Gamma$ (at least if $L$ is chosen correctly), and $\Sigma$ has least area among all surfaces with these properties. We thus obtain a new proof (originally suggested in a paper of Fr\"{o}hlich and Struwe) that the smooth version of Plateau's problem admits a solution for every boundary curve in $\mathbb{R}^3$. This also works if $4 \leq n\leq 7$ and $\Gamma$ is assumed to lie in a strictly convex hypersurface. 
\end{abstract}

\section{Introduction}

Given a $W^{1,2}_{\loc}$-function $u$ of a domain $\Omega$, its Allen--Cahn energy is given by 
    \[E_\varepsilon(u) = \int_\Omega \varepsilon\frac{|\nabla u|^2}{2} + \frac{(1-|u|^2)^2}{4\varepsilon}\,dx.\]
Such double-well functionals were first introduced by physiscists as models for phase-transition phenomena. It is now well known that there is a deep link between these functionals and minimal hypersurfaces. Modica and Mortola \cite{Modica--Mortola} showed that, in the limit as $\varepsilon \to 0$, the Allen--Cahn functional $\Gamma$-converges to the perimeter functional on Cacciopoli sets in $\Omega$. In particular, functions which minimize $E_\varepsilon$ in compact subsets of $\Omega$ exhibit energy concentration on a perimeter-minimizing boundary. Modica later gave a different proof of this property of minimizers in \cite{Modica}. The results of \cite{Modica--Mortola} and \cite{Modica} were extended to sections of a real line bundle (as opposed to functions) over a \emph{closed} Riemannian manifold by Baldo--Orlandi \cite{Baldo--Orlandi}. In that setting, sections which minimize an appropriate analogue of the Allen--Cahn functional concentrate energy on a mass-minimizing $(n-1)$-cycle. Moreover, the cycles which arise in this way all lie in a certain mod~2 homology class determined by the chosen line bundle (namely, the Poincar\'{e}-dual to its first Stiefel--Whitney class). 

In their fundamental work \cite{HT}, Hutchinson and Tonegawa considered general critical points (not just minimizers) of the Allen--Cahn energy. They established that, in the limit as $\varepsilon \to 0$, critical points with uniformly bounded energy exhibit energy concentration on a stationary integer rectifiable $(n-1)$-varifold. Together with the regularity theory for stable critical points developed in \cite{TW}, this opened the door to striking geometric applications, beginning with a new proof of the Almgren--Pitts existence theorem for minimal hypersurfaces in Riemannian manifolds \cite{Guaraco}. For further applications in this vein, see \cite{Gaspar-Guaraco, Chodosh-Mantoulidis, Bellettini_1, Bellettini_2, Bellettini--Workman, Bellettini--Marshall-Stevens, CM_widths}. 

Our first goal in this paper is to generalise the convergence theory of Hutchinson--Tonegawa to sections of real line bundles over the complement of a compact codimension-two submanifold $\Gamma \subset \mathbb{R}^n$, where $n \geq 2$. For $n = 3$ this setting was already considered by Fr\"{o}hlich and Struwe in \cite{Struwe}. They observed that the Allen--Cahn functional on the space of sections of such a bundle ought to be closely related to the area functional on the space of minimal hypersurfaces with boundary $\Gamma$. Our results begin to make precise the link between the variational theories of these two functionals.

Let us fix some notation and terminology. Let $X = M \setminus \Gamma$, where $M$ is either $\mathbb{R}^n$ or an open ball $B \subset \mathbb{R}^n$. Let $L$ be a real line bundle over $X$, equipped with a bundle metric $\langle \cdot, \cdot \rangle $ and a flat metric connection $\nabla$. Given a section $u \in W^{1,2}_{\loc}(X, L)$, we define its Allen--Cahn energy to be
    \begin{equation}\label{energy}
    E_\varepsilon(u) = \int_X \varepsilon \frac{|\nabla u|^2}{2} + \frac{W(u)}{\varepsilon}\,dx, \qquad W(u) := \frac{1}{4}(1-|u|^2)^2.
    \end{equation}
More generally, the energy of $u$ in a subset $\Omega$ will be denoted $E_\varepsilon(u, \Omega)$. We say that $u$ is critical for $E_\varepsilon$ if firstly $E_\varepsilon(u,K) < \infty$ for every compact $K \subset X$ and secondly
    \[\frac{d}{dt}\bigg|_{t=0} E_\varepsilon(u + t\varphi) = 0\]
for every section $\varphi \in C^1_0(X, L)$. Equivalently, $u$ is critical if it weakly solves the Euler--Lagrange equation  
    \[\varepsilon^2 \Delta u = (|u|^2 - 1)u.\]
This equation is subcritical (the right-hand side is $O(|u|^p)$ for $p < \frac{n+2}{n-2}$), so standard elliptic regularity theory implies that every critical section is in fact smooth.

So far our discussion has concerned an arbitrary line bundle $L \to X$. From now on we will always work with a specific choice of nontrivial $L$. When we study the local behaviour of critical sections near $\Gamma$, it will suffice to work in balls $M=B$ such that $B \cap \Gamma$ is close to an $(n-2)$-plane in $B$. In this case, there is exactly one nontrivial line bundle over $X$. When working globally, however, with $M = \mathbb{R}^n$, we allow $\Gamma$ to have more than one connected component, and in this situation there may be many nontrivial line bundles over $X$. However, only one of them detects every component of $\Gamma$. We recall in Section~\ref{topology} that the space of line bundles over $X$ is in bijection with the elements of $\operatorname{Hom}(H_1(X),\Z_2)\simeq H^1(X,\Z_2)$. Among these there is a unique bundle such that the nodal set of any smooth section necessarily intersects $B_r(p)$ for every $p \in \Gamma$ and $r > 0$; we call this the \emph{spanning bundle} over $X$. Equivalently, the spanning bundle corresponds to the element of $H^1(X,\mathbb{Z}_2)$ which sends every set of generators for $H_1(X)$ to the identity. When working globally, with $M = \mathbb{R}^n$, we always take $L \to X$ to be the spanning bundle. We will show that, as $\varepsilon \to 0$, critical sections of this bundle with bounded energy concentrate their energy on a stationary $(n-1)$-varifold whose support contains $\Gamma$. This is to be expected since, for small $\varepsilon$, most of the energy of a critical section lies in a small neighbourhood of its nodal set, and for sections of the spanning bundle any such neighbourhood contains $\Gamma$. 

Given a critical section $u$ for $E_{\varepsilon}$, we define a corresponding energy measure
    \[\mu := \varepsilon \frac{|\nabla u|^2}{2} + \frac{W(u)}{\varepsilon}\,dx.\]
Following \cite{HT}, we also associate with $u$ a diffuse $(n-1)$-varifold $V$ on $M$, by taking a weighted average over smooth level sets of $|u|$. We defer the precise definition of $V$ to Section~\ref{critical sections}, but record here that the mass of $V$ is given by 
    \[\|V\| = \frac{1}{\sigma} \sqrt{\frac{W(u)}{2}} |\nabla u| \, dx,\]
where $2\sigma$ is the Allen--Cahn energy of the 1-dimensional heteroclinic solution $x \mapsto \tanh(\frac{x}{\varepsilon \sqrt{2}})$. 

Our first main result is the following. 

\begin{theorem}\label{main critical}
Let $B \subset \mathbb{R}^n$ be an open ball and set $X = B \setminus \Gamma$. We assume $X$ is diffeomorphic to the complement of an $(n-2)$-plane in $\mathbb{R}^n$, and denote by $L$ the nontrivial real line bundle over $X$. Fix a sequence $\varepsilon_k \to 0$, and for each $k$ suppose $u_k$ is a critical section of $L$ for the energy $E_{\varepsilon_k}$. In addition, suppose there are constants $C_0$ and $C_1$ such that
\begin{equation}\label{uniform bounds}
    \sup_k \bigg(\sup_X |u_k|\bigg) \leq C_0, \qquad \sup_k E_{\varepsilon_k}(u_k, X) \leq C_1.
\end{equation}
Possibly after passing to a subsequence, we have the following behaviour:
\begin{enumerate}[(i)]
    \item The varifolds $V_k$ associated to the sequence $u_k$ weak*-converge to an integer rectifiable $(n-1)$-varifold $V$ on $B$. Moreover, $V$ is stationary with respect to vector fields $g \in C^1_0(B, \mathbb{R}^n)$ such that $g|_\Gamma$ is tangent to $\Gamma$.
    \item The renormalised energy measures $\mu_k/2\sigma$ weak*-converge to $\|V\|$. 
    \item $\Gamma \cap B \subset \supp\|V\|$.
    \item The density of $\|V\|$ satisfies 
        \begin{equation}\label{density intro} C^{-1} \leq \Theta^{n-1}(\|V\|,p)\leq C\bigg(1+\frac{1}{\dist(p,\partial B)^{n-1}}\bigg)\end{equation}
    for every $p \in \supp\|V\|$, where $C=C(n,\Gamma,C_0, C_1)$. In particular, $\|V\|(\Gamma) = 0$. 
    \item\label{hausdorff} For each $b \in (0,1)$, the sublevel sets $\{|u_k| \leq 1-b\}$ Hausdorff-converge to $\supp\|V\|$ in every compact subset of $B$. Moreover, $|u_k| \to 1$ in $C^0_{\loc}(X \setminus \supp\|V\|)$.
\end{enumerate}
\end{theorem}

By a covering argument, Theorem~\ref{main critical} and the interior theory developed in \cite{HT} imply a corresponding global statement for critical sections of the spanning bundle.

\begin{theorem}\label{main critical global}
Suppose we are in the setting of Theorem~\ref{main critical}, but where the $u_k$ are now sections of the spanning bundle over $X = \mathbb{R}^n\setminus \Gamma$. We have all of the same conclusions, but with $B$ replaced by $\mathbb{R}^n$, and where the term involving $\dist(p, \partial B)$ does not appear in \eqref{density intro}. In addition, $\supp\|V\|$ is compact. 
\end{theorem}

With these theorems established, we turn to the regularity of the energy concentration set $\supp\|V\|$. We restrict ourselves to the case of sections which are stable or have bounded Morse index. Let us define these terms. Given a critical section $u$ for $E_\varepsilon$, the second variation of $E_\varepsilon$ at $u$ is given by
    \[\frac{d^2}{dt^2}\bigg|_{t=0} E_\varepsilon(u + t\varphi, X) = \int_\Omega \varepsilon |\nabla \varphi|^2 + \frac{2}{\varepsilon}\langle u, \varphi\rangle^2 + \frac{(|u|^2 - 1)}{\varepsilon}|\varphi|^2\]
for each $\varphi \in W^{1,2}_{0}(X,L)$. For each precompact open $\Omega \subset X$, the subspace of sections in $W^{1,2}_{0}(\Omega,L)$ such that
    \[\frac{d^2}{dt^2}\bigg|_{t=0} E_\varepsilon(u + t\varphi) < 0\]
is finite-dimensional, and its dimension is the Morse index of $u$ in $\Omega$, denoted $\ind(u, \Omega)$. For an arbitrary open $\Omega \subset X$, we define $\ind(u, \Omega)$ to be the supremum over $\ind(u,\Omega')$ for all precompact open $\Omega' \subset \Omega$. We say that $u$ is stable in $\Omega$ if $\ind(u, \Omega) = 0$. If $u_k$ is as in Theorem~\ref{main critical}, and in addition $\ind(u_k, \Omega)$ is bounded indpendently of $k$, then optimal regularity of $\supp\|V\|$ in $\Omega$ follows from \cite{TW} and \cite{Guaraco}. Alternatively, when $n =3$, smoothness of $\supp\|V\|$ can be deduced from \cite{Chodosh-Mantoulidis}. 

\begin{corollary}\label{stable regularity}
Let $u_k$ and $V$ be as in Theorem~\ref{main critical} and fix an open subset $\Omega \subset X$. Suppose $\ind(u_k, \Omega)$ is bounded independently of $k$. Then $\supp \|V\| \cap \Omega$ is a smooth embedded minimal hypersurface outisde a set of Hausdorff dimension at most $n-8$. 
\end{corollary}

We will see that, when $n = 3$, if the sequence $u_k$ is minimizing then $\supp\|V\|$ is also smooth at boundary points. We do not know whether the same is true for sequences of stable sections, but this seems unlikely without further conditions on the boundary curve. However, when $n=3$, for sections with bounded Morse index we show that the structure of $V$ at boundary points is severely restricted. To state the result we note that, as a consequence of \eqref{density intro}, at each point in its support the limiting varifold $V$ admits a varifold tangent (this is proven below in Lemma~\ref{tangent cone}). 

\begin{theorem}\label{stable boundary cone}
Let $u_k$ and $V$ be as in Theorem~\ref{main critical}. Suppose in addition that $n = 3$ and that $\ind(u_k, X)$ is bounded independently of $k$. Consider a point $p \in \Gamma$. Then every varifold tangent to $V$ at $p$ is of the form $\sum_{j = 1}^N m_j V_{P_j}$ for some $N \in \mathbb{N}$, where the $P_j$ are haflplanes meeting along $T_p \Gamma$, $V_{P_j}$ is the unit-multiplicity varifold induced by $P_j$, and the $m_j$ are positive integers such that $\sum_{j=1}^N m_j$ is odd.
\end{theorem}

Note that, by Allard's boundary regularity theorem \cite{Allard_boundary}, when $N = 1$ in Theorem~\ref{stable boundary cone}, $\supp\|V\|$ is a smooth surface with boundary in a neighbourhood of $p$. This fact will play a role in our results concerning \emph{minimizing} sections and Plateau's problem, to which we now turn. 

\subsection{Minimizing sections and Plateau's problem} Plateau's problem is to establish the existence of a surface whose area is minimal among all those which span a given curve in $\mathbb{R}^3$. Lagrange posed the problem in 1760, and Plateau later demonstrated that solutions arise experimentally as soap films clinging to a wire frame. Plateau's problem is far more subtle than it appears; to even formulate it precisely, appropriate notions of surface and area must be decided upon, and one must specify what it means for a surface to span a given curve. Over the last century a number of frameworks have been developed to solve Plateau's problem and its generalisations to higher dimensions. We refer to Section~2 of \cite{DeLellis} for an extensive overview of these different approaches, and point to the particularly relevant references \cite{Douglas, Rado, Reifenberg, Federer-Fleming, Fleming, Hardt-Simon, Harrison-Pugh, Almgren, DGM, KMS}.

Theorems~\ref{main critical global} and \ref{stable boundary cone} allow us to complete a new proof (originally suggested by Fr\"{o}hlich and Struwe in \cite{Struwe}) that the \textit{smooth version} of Plateau's problem admits a solution for every smooth closed curve $\Gamma$ in $\mathbb{R}^3$. That is, we show there is a surface $\Sigma$, smooth up to its boundary, such that $\partial \Sigma = \Gamma$ and the $2$-dimensional Hausdorff measure of $\Sigma$ is minimal among all surfaces in this class. This result was originally established using the theory of flat chains mod~2 \cite{Fleming} and Allard's regularity theorems \cite{Allard, Allard_boundary}. In the approach taken here, we instead consider sections of the spanning bundle over $X = \mathbb{R}^3 \setminus \Gamma$ which minimize the Allen--Cahn energy $E_\varepsilon$. We demonstrate that, in the limit as $\varepsilon \to 0$, the desired Plateau solution arises as the set where these minimizers concentrate energy. We still rely on Allard's theorems to prove boundary regularity, but we expect this to be achievable using PDE arguments (as in the interior case \cite{Wang-Wei_a, Chodosh-Mantoulidis}) leading to a resolution of Plateau's problem that relies only on level-set estimates for solutions of semilinar elliptic equations.

Let us now state our results concerning Plateau's problem in more detail. A section $u \in W^{1,2}_{\loc}(X, L)$ minimizes the energy $E_\varepsilon$ if 
    \[E_\varepsilon(u) = \inf \big\{ E_{\varepsilon}(v) : v \in W^{1,2}_{\loc}(X, L)\big\}.\]
Every minimizing section is also critical and hence smooth. In all dimensions, a minimizing section exists for every $\varepsilon > 0$ (this was proven in \cite{Struwe}, and we recap the argument in Section~\ref{minimizers}).

\begin{theorem}\label{Plateau theorem}
Consider a smooth closed curve $\Gamma \subset \mathbb{R}^3$. Fix a sequence $\varepsilon_k \to 0$, and let $u_k$ denote a sequence of sections of the spanning bundle over $X = \mathbb{R}^3\setminus\Gamma$ which minimize $E_{\varepsilon_k}$. The bounds \eqref{uniform bounds} then hold automatically, so (after passing to a subsequence) we may apply Theorem~\ref{main critical global} to extract a weak*-limit of the associated varifolds, denoted $V$. The set $\Sigma := \supp\|V\|$ is a smooth surface with boundary, $\partial \Sigma = \Gamma$, and $V$ is the unit-multiplicity varifold induced by $\Sigma$. Moreover, $\Sigma$ is area minimizing with respect to $\Gamma$. That is, $\Sigma$ solves the smooth version of Plateau's problem for $\Gamma$. 
\end{theorem}

The smooth Plateau problem can also be posed in higher dimensions, where $\Gamma$ is now a compact codimension-two submanifold, and one seeks an area minimzing hypersurface $\Sigma$ with boundary $\partial \Sigma = \Gamma$. However even for $n = 4$ there may be no solutions. Suppose for example that $\Gamma$ is an embedded $\mathbb{RP}^2$ in $\mathbb{R}^4$. Since $\mathbb{RP}^2$ is not the boundary of any smooth 3-manifold, in this case $\Sigma$ must be singular on some subset of $\Gamma$. The situation is different, however, if $\Gamma$ lies in a strictly convex hypersurface---in this case there is a smooth solution if $4 \leq n \leq 7$ (see \cite[5.2]{Allard_boundary}). We also obtain a new proof of this statement. 

\begin{theorem}\label{Plateau theorem high dim}
Let $\Gamma$ denote a closed codimension-two submanifold of $\mathbb{R}^n$, $n \geq 4$. Fix a sequence $\varepsilon_k \to 0$, and let $u_k$ denote a sequence of sections of the spanning bundle over $X = \mathbb{R}^n \setminus \Gamma$ which minimize $E_{\varepsilon_k}$. The bounds \eqref{uniform bounds} then hold automatically, so (after passing to a subsequence) we may apply Theorem~\ref{main critical global} to extract a weak*-limit of the associated varifolds, denoted $V$. Let $\Sigma := \supp\|V\|$. We then have that $\Sigma \setminus \Gamma$ is a smooth minimal hypersurface outside a set of Hausdorff-dimension at most $n-8$, and $V$ is the unit-multiplicity varifold induced by $\Sigma$. Moreover, if $\Sigma'$ is any smooth hypersurface with boundary such that $\partial \Sigma' = \Gamma$, then 
    \[\mathcal H^{n-1}(\Sigma) \leq \mathcal H^{n-1}(\Sigma').\]
If we additionally assume that $\Gamma$ lies in a strictly convex hypersurface, then there is a tubular neighbourhood of $\Gamma$ in which $\Sigma$ is a smooth hypersurface with boundary, and $\partial \Sigma = \Gamma$. Consequently, if $4 \leq n \leq 7$ and $\Gamma$ lies in a strictly convex hypersurface, then $\Sigma$ is a smooth hypersurface with boundary which solves Plateau's problem for $\Gamma$. 
\end{theorem}

\subsection{Key steps in the proofs} Let us first describe some of the main steps in the proof of Theorem~\ref{main critical}. Here much of our analysis follows the work of Hutchinson and Tonegawa \cite{HT}, who dealt with the interior case, but there are key differences at the boundary. 

Suppose we are in the setting of Theorem~\ref{main critical}, and to ease notation let $\varepsilon_k = \varepsilon$ and $u = u_k$. In Section~\ref{critical sections}, we first recall (and generalise to higher dimensions) certain estimates from \cite{Struwe} which provide control on $u$ and its derivatives in a small tubular neighbourhood of $\Gamma$. In particular, writing $\rho(x)$ for the distance from $x$ to $\Gamma$, we have 
    \begin{equation}\label{bdy derivs intro}|u(x)|^2 \lesssim \varepsilon^{-1}\rho(x) \qquad \text{and} \qquad |\nabla u|^2 \lesssim \varepsilon^{-1} \rho(x)^{-1} \end{equation}
at points where $\rho(x) \lesssim \varepsilon$. Notice that the second estimate allows $|\nabla u|^2$ to blow up at a controlled rate as we approach $\Gamma$. These boundary estimates are then used to derive an almost-monotonicity formula for the rescaled energy
    \[\frac{1}{r^{n-1}}\int_{B_r(p)} \varepsilon \frac{|\nabla u|^2}{2} + \frac{W(u)}{\varepsilon}\,dx\]
in balls centered at a point $p \in \Gamma$. A similar formula for interior balls was derived in \cite{HT}. As in \cite{HT}, we find that the rescaled energy is not quite monotone in $r$, because of a term involving the \emph{discrepancy},  
    \[\xi := \varepsilon \frac{|\nabla u|^2}{2} - \frac{W(u)}{\varepsilon}\]

As a consequence of interior estimates proven in \cite{HT}, the discrepancy is bounded from above on compact subsets of $X$, and decays to $0$ in $L^1_{\loc}(X)$ as $\varepsilon \to 0$. To prove Theorem~\ref{main critical} we need more; it will be necessary to show that $\xi$ decays to $0$ in $L^1_{\loc}(B)$. This global statement is more delicate than its interior counterpart because, in our setting with boundary, $\xi$ is typically unbounded from above near $\Gamma$ for each $\varepsilon$ (see the example at the conclusion of this introduction). Despite this, by carefully combining an improved interior estimate for $\xi$ with the boundary estimates \eqref{bdy derivs intro}, we obtain the necessary $L^1_{\loc}$-decay over $B$. This ensures that the almost-monotonicity improves to a genuine monotonicity in the limit as $\varepsilon\to 0$. 

With this fact in hand, in Section~\ref{varifolds} we show that the diffuse varifolds associated with the sequence $u$ approach a limit which is stationary with respect to vector fields tangent to $\Gamma$ and, moreover, satisfies uniform upper and lower density bounds. The limiting varifold is then rectifiable by Allard's rectifiability theorem, and integral by an argument in \cite{HT}. This concludes our discussion of Theorem~\ref{main critical}. As mentioned above, Theorem~\ref{main critical global} follows from Theorem~\ref{main critical} and a covering argument. 

Theorem~\ref{stable boundary cone} is proven in Section~\ref{section dimension 3}. In the setting of that theorem, it is straightforward to show that every varifold tangent to $V$ at $p \in \Gamma$ is supported on a smooth minimal cone away from the tangent line $T_p \Gamma$. Since we are in ambient dimension 3, such a cone can only consist of a union of halfplanes meeting along $T_p \Gamma$, and we know that each of these occur with integer multiplicity by Theorem~\ref{main critical}. Showing that the sum of the multiplicities is odd is the most difficult step in the proof. Using the $C^{2}$-estimate for nodal sets of stable solutions due to Chodosh and Mantoulidis \cite{Chodosh-Mantoulidis}, we are able to show that the sum of the multiplicities is equal to the number of zeroes of $u_k$ on a small loop winding once around $\Gamma$, at least when $k$ is large. This number is odd for topological reasons---for example, one can see this by noting that the pullback of $L$ to the loop is the M\"{o}bius bundle. 

We now summarise the proofs of Theorem~\ref{Plateau theorem} and Theorem~\ref{Plateau theorem high dim}. These are given in Section~\ref{minimizers}. First, we use the direct method to show that for every $\varepsilon > 0$ and compact $(n-2)$-dimensional boundary $\Gamma$, there exists a smooth section $u$ of the spanning bundle $L$ over $X = \mathbb{R}^n \setminus \Gamma$ which minimizes $E_\varepsilon$. We show that these minimizers satisfy $|u| \leq 1$, and that their energy remains bounded as we send $\varepsilon \to 0$. Theorem~\ref{main critical global} can than be invoked to conclude that, after passing to a subsequence, the varifolds associated to $u$ weak*-converge to an integer rectifiable limit $V$. Moreover, the analysis we conduct in Section~\ref{critical sections} shows that $\Sigma := \|V\|$ contains $\Gamma$. The set $\Sigma \setminus \Gamma$ is smooth away from a set of Hausdorff dimension at most $n-8$ by Corollary~\ref{stable regularity}. 

Our claim concerning boundary regularity is that $\Sigma$ is a smooth hypersurface with boundary in a tubular neighbourhood of $\Gamma$, provided $\Gamma$ lies in a strictly convex hypersurface or $n = 3$. If $\Gamma$ lies in a strictly convex hypersurface, we can simply argue as in \cite[Section~5]{Allard_boundary}. For a general boundary curve in $\mathbb{R}^3$ the argument is more involved. By Theorem~\ref{stable boundary cone}, we know that for each $p \in \Gamma$, the limit $V$ has a varifold tangent supported on $N$ halfplanes. Since $V$ arises from a sequence of minimizers, a standard cut-and-paste argument shows that each of these occurs with unit multiplicity, so by Theorem~\ref{stable boundary cone}, $N$ is odd. We argue that an array of 3 or more halfplanes cannot arise from a sequence of minimizers, and so conclude that $N =1$. It then follows that $\Sigma$ is smooth in a neighbourhood of $p$, as a consequence of Allard's boundary regularity theorem \cite{Allard_boundary}. 

To complete the proofs of Theorem~\ref{Plateau theorem} and Theorem~\ref{Plateau theorem high dim}, it only remains to show that the support of the limiting varifold obtained from a sequence of minimizers has no more area than any smooth hypersurface with boundary $\Gamma$. This is easily proven by contradiction---if there is a competitor with less area, then the original sequence of sections could not have been minimizing, since one can construct a different sequence which has less energy. 

\subsection{An example} We conclude this introduction with the construction of a critical section for $E_\varepsilon$ which serves as an illustrative example; its nodal set is a half-line in $\mathbb{R}^2\setminus\{0\}$, and its discrepancy is unbounded from above. 

Let $L$ be the unique nontrivial line bundle over $\mathbb{R}^2\setminus\{0\}$. For each large $r > 0$ and small $\delta > 0$ we define domains
    \[A_{r, \delta} := \{\rho e^{i\theta} \in \mathbb{R}^2 : r^{-1} \leq \rho \leq r, \; |\theta| \geq \delta\}, \qquad A_r := A_{r,0}.\]
For each $\varepsilon > 0$ there is a minimizer of the Allen--Cahn energy in $W^{1,2}_0(A_{r,\delta})$. Let $w_{r,\delta}$ denote such a minimizer. Easy comparison arguments show that $0 < w_{r,\delta} \leq 1$ holds almost everywhere in $A_{r,\delta}$. Standard elliptic regularity theory then implies that $w_{r,\delta}$ is smooth away from the corners of $A_{r,\delta}$. In fact we get uniform estimates which allow us to send $\delta \to 0$ and extract a limit $w_r$ in $W^{1,2}_0(A_r)$. The limit is smooth in the interior and vanishes on the the positive $x_1$-axis in $A_r$. It is not difficult to show that $w_r$ is reflection-symmetric through the $x_1$-axis---this follows by applying the maximum principle to $w_r/\bar w_r$, where $\bar w_r(x_1, x_2) := \bar w_r (x_1,-x_2)$. It follows that we can define a $C^1$-section $u_r$ of $L$ over $A_r$ by choosing unit sections $e_\pm$ for $L$ over $S_{\pm} := \mathbb{R}^2\setminus\{(x_1,0) : \pm x_1 \geq 0\}$ and declaring that 
    \[u_r(x) \cdot e_+ = w_r(x), \;\; x \in S_+, \qquad 
        u_r(x) \cdot e_- = 
        \begin{cases}
            w_r(x), & x \in S_- \cap \{x_2 \geq 0\}\\
            -w_r(x), & x \in S_-\cap \{x_2 < 0\}.
        \end{cases}
    \]
One can then check directly that $u_r$ solves the Euler--Lagrange equation      \[\varepsilon^2 \Delta u_r = (|u_r|^2 - 1)u_r\]
in the weak sense, and hence is smooth in $A_r$. Moreover, we have uniform interior estimates which ensure that, upon sending $r \to \infty$, we obtain a smooth section $u$ of $L$ which is critical for $E_\varepsilon$ and whose nodal set contains the positive $x_1$-axis. In fact, for any ball $B \subset S_+$, we can use the positive minimizer in $W^{1,2}_0(B)$ as a barrier to ensure that $u$ is nonzero in $B$. So the nodal set of $u$ is precisely the positive $x_1$-axis. The energy of $u$ in $B_1(0)\setminus\{0\}$ is bounded uniformly as $\varepsilon \to 0$. This can be seen by comparing with a Lipschitz section which agrees with $u$ on $\partial B_1(0)$, vanishes inside $\{|x|\leq \varepsilon\}$ and equals 1 in $\{2\varepsilon \leq |x| \leq 1-\varepsilon\}$.

By \cite[Proposition~3.3]{HT}, the discrepancy $\xi$ of the solution just constructed satisfies an a~priori upper bound in each ball $\bar B \subset \mathbb{R}^2 \setminus \{0\}$. However, as we will now demonstrate, $\xi(x)$ becomes unbounded from above as $x \to 0$. This is one of the key differences in behaviour exhibited by critical sections at boundary points and in the interior, and is the main difficulty which we must overcome in our proof of Theorem~\ref{main critical}. 
    
To see that $\xi$ is unbounded, it is useful define $v(z) := u(z^2)$, where $z$ is the standard complex coordinate on $\mathbb{R}^2\setminus\{0\}$. The pullback of $L$ by $z \mapsto z^2$ is trivial, so we may view $v$ as a function on $\mathbb{R}^2 \setminus \{0\}$. Straightforward computations show that
    \begin{equation}\label{complex square eq}
        \varepsilon^2 \Delta v = 4|z|^2(|v|^2 - 1) v,
    \end{equation}
and that the discrepancy of $u$ satisfies
    \[\xi(z^2) = \varepsilon \frac{|\nabla v(z)|^2}{8|z|^2} - \frac{(1-|v(z)|^2)^2}{4\varepsilon}.\]
Moreover, since $E_\varepsilon(u) < \infty$, $v$ has finite $W^{1,2}$-norm in $B_1(0)\setminus\{0\}$, and can therefore be extended to a smooth solution of \eqref{complex square eq} in $B_1(0)$, by setting $v(0)=0$. From our construction of $u$ it follows that (up to a choice of signs) $v$ is positive in $\{x_2 > 0\}$ and negative in $\{x_2 < 0\}$. By the Hopf lemma, we conclude that $|\nabla v| > 0$ in a neighbourhood of the origin. From the formula for $\xi$ above we conclude that $\xi(z^2)\to\infty$ as $z \to 0$.

\subsection{Acknowledgements} The authors would like to express gratitude to M. Struwe for suggesting the problem and to T.~Bourni for a number of helpful conversations. 


\section{Real line bundles}\label{topology}

In this section we discuss the natural correspondence between the set of vector bundles of rank $1$ over a connected manifold $X$ and the cohomology group $H^1(X,\Z_2)$. Later on, we delve into the following two cases: \begin{enumerate}
    \item $X=M\setminus \Gamma$, where $H_1(M)=H_2(M)=0$ and $\Gamma\subset M$ is a codimension two embedded submanifold. This includes $M\simeq \R^n$.
    \item $X=M$ is a closed manifold.
\end{enumerate} We focus on these because of the applications we have in mind: case (1) is motivated by the search for a line bundle that provides a setting for solving the Plateau problem with boundary $
\Gamma$, while case (2) allows one to generate minimal hypersurfaces in arbitrary homology classes in closed manifolds. However, as we later discuss, some of the arguments for case (1) transfer to more general situations. 

\subsection{Real line bundles and $\operatorname{Hom}(H_1(M),\Z_2)\simeq H^1(M,\Z_2)$}

Given a real line bundle $\pi: L \to X$, we can use a partion of unity to equip $L$ with a bundle metric. This allows us to define $U(L) \subset L$, the set of unitary vectors with respect to such metrics. The restriction $\pi: U(L) \to X$ is a double cover of $X$, which is independent of the choice of metric modulo homeomorphisms. It is a simple exercise to see that $U(L)$ is connected if and only if $L$ is not the trival bundle. This observation is the first half of the correspondence between nontrivial line bundles over $X$ and connected double covers of $X$. In the other direction, a connected double cover $\widetilde{X}\to X$ is always a normal covering space whose group of deck transformations is $\Z_2$. In particular, such a covering is generated by a homemorphism $\rho: \widetilde{X}\to\widetilde{X}$, such that $\rho^2=\operatorname{id}|_{\widetilde{X}}$. It is another simple exercise to see that the quotient $L=\widetilde{X}\times \R/\sim$, where $(p,t)\sim (\rho(p),-t)$, is a nontrivial line bundle over $X$, and that the two constructions above are inverse modulo homeomorphisms. 

By the classification of covering spaces, connected double covers of $X$ are in correspondence with $\{H<\pi_1(X):[H:\pi_1(X)]=2\}$, i.e., the set of subgroups of $\pi_1(X)$ of index 2. Index two subgroups are normal, so the first isomorphism theorem for groups implies that such subgroups correspond to elements of $\operatorname{Hom}(\pi_1(X),\Z_2)$. Since $\Z_2$ is abelian, these homorphisms factor through the abelianisation of $\pi_1(X)$, which is $H_1(X)$. This gives the correspondence 
    \[\operatorname{Hom}(\pi_1(X),\Z_2)\simeq \operatorname{Hom}(H_1(X),\Z_2).\]
Finally, by the universal coefficient theorem for cohomology, we have that 
    \[\operatorname{Hom}(H_1(X),\Z_2)\simeq H^1(X,\Z_2).\]
Putting this all together, we have shown that there are as many real line bundles over $X$ as there are elements in $\operatorname{Hom}(H_1(X),\Z_2)\simeq H^1(X,\Z_2)$. We note that the element of $H^1(X,\Z_2)$ associated with a line bundle is known as its first Stiefel--Whitney class. In many applications, $H_1(X)$ is a finitely generated abelian group and an element of $\operatorname{Hom}(H_1(X),\Z_2)$ is uniquely determined by where it sends a set of generators. The relation between the generators of $H_1(X)$ and the global problem we want to solve in $X$ must be treated in a case by case fashion. 

\subsection{Real line bundles from classes of loops mod $2$}

We begin  with case (1). Let $X= M \setminus \Gamma$, where $M$ is diffeomorphic to $\R^n$ and $\Gamma$ is a closed submanifold of $M$ of codimension $2$. In this case, 
    \[H_1(M\setminus \Gamma)\simeq H^{n-2}(\Gamma)\simeq (\Z\oplus \cdots \oplus \Z) \oplus (\Z_2 \oplus \cdots \oplus \Z_2)\]
where the number of copies of $\Z$ and $\Z_2$ are equal to the number of orientable and non-orientable connected components of $\Gamma$, respectively. Moreover, a generator corresponding to a connected component $\Gamma_i\subset \Gamma$ is given by any loop $\gamma_i$ satisfying $\operatorname{link}_{\Z} (\gamma_i,\Gamma_i)=\pm 1$ and $\operatorname{link}_{\Z}(\gamma_i,\Gamma\setminus \Gamma_i)=0$. It is a simple exercise to check that there is exactly one line bundle $L$ over $X = M \setminus \Gamma$ such that the nodal set of a generic section intersects each generator $\gamma_i$ an odd number of times. Indeed, this is the bundle corresponding to the homomorphism $H_1(M)\to \Z_2$ which sends each loop in a set of generators to $1$. We call this bundle the \emph{spanning bundle} over $X$. The spanning bundle is the only line bundle where we can hope to solve Plateau's problem for $\Gamma$. It is worth noting that the pullback of the spanning bundle to any generator $\gamma_i$ is isomorphic to the unique nontrivial line bundle over $S^1$, namely the M\"{o}bius bundle. 

\begin{remark}
For completeness, we summarise how the first isomorphism above arises as a
composition of several maps. Let $T$ be a closed tubular neighborhood of $\Gamma$. The isomorphism then factors as 
    \[H_1(M\setminus\Gamma) \simeq H_2(M,M\setminus\Gamma) \simeq H_2(T,T\setminus \Gamma)\simeq H_2(T,\partial T)\simeq H^{n-2}(T)\simeq H^{n-2}(\Gamma).\]
The first map is obtained from the long exact sequence of the pair $(M,M\setminus \Gamma)$:  $$\cdots \to 0=H_2(M) \to H_2(M,M\setminus\Gamma) \to H_1(M\setminus\Gamma) \to H_1(M)=0\to \cdots$$ The second map is just the excision property. The third is a retraction of $T\setminus\Gamma$ onto $\partial T$. The fourth is Lefschetz duality, which generalises Poincar\'e duality to orientable manifolds with boundary. Finally, the fifth and last map is the retraction of $T$ onto $\Gamma$. 
\end{remark} 

\subsection{Real line bundles from classes of hypersurfaces mod $2$} Case (2) is a lot simpler. In fact, it follows directly from Poincar\'e duality that $H^1(M,\Z_2)\simeq H_{n-1}(M,\Z_2)$. In particular, for each nontrivial class of $(n-1)$-dimensional cycles mod $2$, there exists a unique surjective homomorphism from $H_1(M)$ to $\Z_2$. By what we disscused at the begining of the section, each one of these homomorphisms gives rise to a real line bundle over $M$. It is then a simple exercise to check that the nodal set of a generic section of this bundle is in the mod $2$ homology class which gave rise to the bundle in the first place. 


\section{A priori estimates for critical sections}\label{critical sections}

Let $L$ be a real line bundle over an open subset $X \subset \mathbb{R}^n$, where $n \geq 2$. We equip $L$ with a bundle metric $\langle \cdot,  \cdot \rangle$ and a flat metric connection $\nabla$. We denote the Euclidean inner product of two vectors $g, h \in \mathbb{R}^n$ by $g \cdot h$, and write $D$ for the Euclidean connection acting on vector fields tangent to $X$. The induced metric and connection on $L \otimes T^*M$ will also be denoted $\langle \cdot, \cdot \rangle$ and $\nabla$ respectively. We define a Laplacian $\Delta$ acting on sections of $L$ by 
    \[\Delta u := \tr(\nabla \nabla u).\]
With respect to an orthonormal frame $\{e_i\}_{i=1}^n$ for $\mathbb{R}^n$ we have 
    \[\langle \nabla u, \nabla \varphi\rangle = \sum_{i=1}^n \langle \nabla_i u, \nabla_i \varphi\rangle, \qquad \Delta u = \sum_{i=1}^{n} \nabla_i(\nabla_i u).\]

Having fixed a bundle $L \to X$ as above, the class of $C^k$ sections with compact support in $X$ will be denoted $C^{k}_{0}(X,L)$. A section lies in the space $W^{1,2}_{\loc}(X, L)$ if $|u| \in L^2_{\loc}(X)$ and there exists a section $\nabla u$ of $L \otimes T^*M$ with the following properties. Firstly, $|\nabla u| \in L^2_{\loc}(X)$ and, secondly, given an orthonormal frame $\{e_i\}_{i=1}^{n}$ for $\mathbb{R}^n$, we have
    \[\int_X \langle \nabla_i u, \varphi\rangle \, dx = - \int_X \langle u, \nabla_i \varphi \rangle \, dx\]
for every $1 \leq i \leq n$ and $\varphi \in C^1_0(X, L)$.  

Given a constant $\varepsilon > 0$ and a section $u \in W^{1,2}_{\loc}(X,L)$ we define the energy of $u$ in each Borel set $A \subset X$ by  
    \[E_\varepsilon(u, A) := \int_{A} \ \varepsilon\frac{|\nabla u|^2}{2} + \frac{W(u)}{\varepsilon} \, dx, \qquad W(u) := \frac{1}{4}(1-|u|^2)^2.\]
The total energy of $u$ is $E_\varepsilon(u) := E_\varepsilon(u, X)$. A section $u \in W^{1,2}_{\loc}(X,L)$ of $L$ is defined to be critical for $E_\varepsilon$ if $E_\varepsilon(u, K) < \infty$ for every compact $K \subset X$ and 
    \[\frac{d}{dt}\Big|_{t = 0} E_\varepsilon(u + t \varphi) = 0,\]
or equivalently 
    \[\varepsilon^2 \int_X \langle \nabla u , \nabla \varphi \rangle = \int_X \langle W'(u), \varphi \rangle, \qquad W'(u) := (|u|^2 - 1)u,\]
for every $\varphi \in C^{1}_{0}(X,L)$. Standard elliptic theory implies that any critical section $u$ is smooth, and hence satisfies
    \[\varepsilon^2 \Delta u = W'(u).\]

Let $\Gamma$ be a closed codimension-two submanifold of $\mathbb{R}^n$.
In this section we always assume either that $X = \mathbb{R}^n \setminus \Gamma$, or else
    \begin{gather}  X = B \setminus \Gamma, \text{ where $B \subset \mathbb{R}^n$ is an open ball, and} \notag \\ \text{$B \setminus \Gamma$ is diffeomorphic to the complement of an $(n-2)$-plane in $\mathbb{R}^n$.} \label{local} \tag{A1} \end{gather}
For $L \to X$ nontrivial and positive constants $C_0$ and $C_1$, we are interested in the situation where
    \begin{gather} u_k \text{ is a sequence of critical sections for } E_{\varepsilon_k} \text{ such that } \varepsilon_k \to 0 \text{ and } \notag\\
    \sup_k \bigg(\sup_X |u_k|\bigg) \leq C_0, \qquad \sup_k E_{\varepsilon_k}(u_k) \leq C_1. \label{sequence of sections} \tag{A2}
    \end{gather}
Our goal in this section and the next is to study how the energy of such a sequence concentrates as $\varepsilon_k \to 0$, and in particular to prove Theorems~\ref{main critical} and \ref{main critical global}. We mostly work in the local setting \eqref{local}, in which there is only one nontrivial line bundle $L \to X$. Much of our analysis is inspired by \cite{HT}.

\subsection{Derivative estimates at the boundary} Suppose we are in the setting of \eqref{local}, and denote by $L$ the nontrivial real line bundle over $X$. We recall some key estimates from \cite{Struwe} concerning the pointwise behaviour of critical sections for $E_\varepsilon$ near $\Gamma$. As we pointed out in the introduction, even if $u$ is bounded, its derivative $\nabla u$ may not be. However we have some control on the rate at which $\nabla u$ may blow up as we approach $\Gamma$.  

Let $\rho$ denote the distance function to $\Gamma$. There is a positive constant $\delta_0(\Gamma)$ such that $\Gamma$ admits a tubular neighbourhood of size $\delta_0$. This tubular neighbourhood is given by $\{\rho < \delta_0\}$, and $\rho$ is smooth there. Given a vector field $g$ defined in $\{\rho < \delta_0\}$, let us write
    \[g^\perp := (g \cdot \grad \rho) \grad \rho, \qquad g^\top := g - g^\perp.\]

\begin{proposition}\label{boundary derivatives}
Let $u \in C^\infty(X,L)$ be a critical section for $E_\varepsilon$ such that $|u| \leq C_0$ and $E_\varepsilon(u) < \infty$. There exists a constant $\delta_1(n,\Gamma) < \delta_0$ with the following property. If $\varepsilon R < \delta_1$ then for any nonnegative integers $m = k + l$, and unit vectors $\{e_i\}_{i=1}^k$ and $\{f_i\}_{i=1}^l$ such that $e_i^\top = 0$ and $f_i^\perp = 0$, we have
    \[|\nabla^{m} u(e_1, \dots, e_k, f_1, \dots, f_l)|^2 \leq C \varepsilon^{-1-2l} \rho^{1-2k}\]
at every point $x \in X$ satisfying $\rho(x) \leq R\varepsilon$ and $\dist(x, \partial B) \geq 2R_\varepsilon$, where $C = C(n, \Gamma, C_0,R, m)$. In particular, at every such $x \in X$ we have
    \[|u|^2 \leq C\varepsilon^{-1}\rho \qquad \text{and} \qquad |\nabla u|^2 \leq C \varepsilon^{-1} \rho^{-1},\]
where where $C = C(n, \Gamma, C_0,R)$. 
\end{proposition}

Proposition~\ref{boundary derivatives} was proven for $n =3$ in \cite[Theorem~3.2]{Struwe}. The strategy is to conisder a small normal tubular neighbourhood of $\Gamma$, and lift $u$ by a map which covers this neighbourhood twice (the map acts by the complex square on normal discs). The lift of $u$, denoted $\tilde u$, can be represented by a function which is odd in normal directions. Finiteness of energy ensures that $\tilde u$ extends to a weak solution of the lifted Euler--Lagrange equation for all odd test functions, not just those which compactly supported away from $\Gamma$. Using this fact, one can derive $L^2$ derivative estimates for $\tilde u$ up to any order, and deduce pointwise bounds using the Sobolev embedding theorem. Passing these to the covering map yields the estimates for $u$ stated in Proposition~\ref{boundary derivatives}.

The same argument works in dimensions $n \geq 4$, except that some extra lower-order terms appear in the computations. The reason is that $N\Gamma$ may not admit a parallel orthonormal frame, even locally. However these extra terms can all be absorbed in a straightforward manner. We describe how this works, up to the point where the argument in \cite{Struwe} applies.

Let $\delta_1 < \delta_0$ be a small positive constant to be refined as we proceed. Fix an $\varepsilon >0$, and suppose $R$ is such that $R \varepsilon < \delta_1$. We consider an arbitrary point in $p \in \Gamma$ such that $\dist(p, \partial B) \geq 2R\varepsilon$. Let $u$ be a critical section for $E_\varepsilon$ such that $|u| \leq C_0$ and $E_\varepsilon(u) < \infty$.

It will be convenient to rescale. We define $D(x) := \tfrac{x-p}{R\varepsilon}$ and set
    \[\hat B = D(B), \qquad \hat \Gamma := D(\Gamma), \qquad \hat X := \hat B \setminus \hat \Gamma,\]
and denote by $\hat L$ the pullback of $L$ to $\hat X$ via $D^{-1}$. We equip $\hat L$ with the pullback metric and connection, still denoted $\langle \cdot, \cdot \rangle$ and $\nabla$ respectively. Let $\hat u$ denote the pullback of $u$, namely $\hat u(x) = u(R\varepsilon x+p)$. The section $\hat u$ is critical for $E_{\hat \varepsilon}$, where $\hat \varepsilon := R^{-1}$. Moreover, $\hat u$ satisfies $|\hat u| \leq C_0$ and 
    \[E_{\hat \varepsilon}(\hat u) = \frac{1}{(R\varepsilon)^{n-1}} E_\varepsilon(u) <\infty.\]
By taking $\delta_1$ sufficiently small we can assume that $\hat \Gamma$ is as close as we like in $C^\infty$ to the linear subspace $T_p\Gamma$ in any neighbourhood of the origin. In particular, we may then choose Gaussian coordinates for the unit tubular neighbourhood of $\hat \Gamma$ near 0, via a map $\Phi : B^2_{1} \times B^{n-2}_1 \to \mathbb{R}^{n}$ defined as follows:    
    \begin{itemize}
        \item The restriction of $\Phi$ to $\{0\} \times B^{n-2}_1$ gives a system of normal coordinates on $\hat \Gamma$.
        \item For all $y = (z_1,z_2,x)$ in $B^2_{1} \times B^{n-2}_1$,
            \[\Phi(y) = z_1 \nu_1(\Phi(0,x)) + z_2\nu_2(\Phi(0,x)) +  \Phi(0,x),\]
        where $\{\nu_1, \nu_2\}$ is a local orthonormal frame for $N\hat \Gamma$ which is parallel (with respect to the induced connection on $N \hat \Gamma$) along every radial geodesic emanating from 0. 
    \end{itemize}
We write $\Omega = (B_{1}^{2}\setminus\{0\}) \times B_1^{n-2}$. Let $S:\Omega \to \Omega$ be defined by
    \[S(z_1,z_2,x) := (z_1^2 - z_2^2, 2z_1z_2, x),\]
so that $S$ acts on $z$ by taking the complex square. We define $P = S \circ \Phi$ and consider the pullback bundle $P^* \hat L$, equipped with the pullback metric and connection, still denoted $\langle \cdot, \cdot \rangle$ and $\nabla$ respectively. It is not difficult to show that $P^* \hat L$ is trivial, and hence admits a globally defined unit section, which we consider to be fixed from now on. Sections of $\hat L$ over $\Phi(\Omega)$ are then in one-to-one correspondence, via the pullback $P^*$, with functions on $\Omega$ which are odd in $z$. In particular, we may view $\tilde u := P^* \hat u$ as a function on $\Omega$ which satisfies $\tilde u(z,x) = -\tilde u(-z,x)$.

Let $\varphi$ denote an arbitrary section in $C^1_0(\Phi(\Omega), \hat L)$, and define $\tilde \varphi := P^* \varphi$. We write $g$ for the pullback of the Euclidean metric on $\mathbb{R}^n$ via $P$, so that 
    \[g_{ij}(z,x) :=  \partial_{y_i} P(z,x) \cdot \partial_{y_j} P(z,x),\]
and write $g^{ij}$ for the inverse matrix of $g_{ij}$. We then have 
    \[\langle \hat u(P(z,x)), \varphi(P(z,x))\rangle = \tilde u(z,x) \tilde \varphi(z,x),\] 
and
    \[\langle \nabla \hat u (P(z,x)) , \nabla \varphi(P(z,x)) \rangle = g^{ij}(z,x)\partial_{y_i} \tilde u (z,x) \partial_{y_j} \tilde \varphi (z,x).\]
Since $\hat u$ is critical for $E_{\hat \varepsilon}$,
    \[0 = \int_{\Phi(\Omega)} \langle \nabla \hat u , \nabla \varphi \rangle + \hat \varepsilon^{-2}(|\hat u|^2 - 1) \langle \hat u , \varphi \rangle \,dx.\]
Equivalently, $\tilde u$ satisfies 
    \begin{align}\label{z squared weak y}
        0 =\int_{\Omega} \Big(g^{ij}\partial_{y_i} \tilde u \partial_{y_j} \tilde \varphi + \hat \varepsilon^{-2}(|\tilde u|^2 - 1) \tilde u \tilde \varphi\Big) \sqrt{|\det(g)|}\,dy.
    \end{align}
As $\delta_1 \to 0$, $\Phi$ converges smoothly to the identity, so $P$ converges smoothly to $S$. A straightforward computation shows that we may write 
    \[g = \begin{bmatrix}4|z|^2 & 0 & 0\\0 & 4|z|^2 & 0\\ 0 & 0 & 1\end{bmatrix} + a(z,x),\]
where $|z|^{-2}a(z,x) \to 0$ in $C^\infty$, uniformly in $x$, as $\delta_1 \to 0$. It follows that
    \[\sqrt{|\det(g)|} = 4|z|^2h(z,x), \qquad g^{-1} = h(z,x)^{-1}\begin{bmatrix}\frac{1}{4|z|^2} & 0 & 0\\0 & \frac{1}{4|z|^2} & 0\\ 0 & 0 & 1\end{bmatrix} + b(z,x)\]
where where $h(z,x) \to 1$ and $|z|^2b(z,x) \to 0$ in $C^\infty$, uniformly in $x$, as $\delta_1 \to 0$. With this notation \eqref{z squared weak y} becomes
    \begin{equation}\label{z squared weak}
        0 = \int_{\Omega} \partial_{z_i} \tilde u \partial_{z_i} \tilde \varphi + 4|z|^2 \Big(\partial_{x_i} \tilde u \partial_{x_i} \tilde \varphi + h b^{ij}\partial_{y_i} \tilde u \partial_{y_j} \tilde \varphi + \hat \varepsilon^{-2} h (|\tilde u|^2 - 1) \tilde u \tilde \varphi\Big) \, dz dx.
    \end{equation}
Moreover, the energy of $u$ in $\Phi(\Omega)$ can be expressed as 
    \begin{align*}
        \int_{\Phi(\Omega)} \hat \varepsilon & \frac{|\nabla \hat u|^2}{2} + \frac{W(\hat u)}{\hat \varepsilon} \, dx\\
        &= \frac{1}{2}\int_{\Omega} \hat \varepsilon \frac{|\partial_{z} \tilde u|^2}{2} + 4|z|^2 \bigg(\hat \varepsilon \frac{|\partial_{x} \tilde u|^2}{2} + \hat \varepsilon h b^{ij}\frac{\partial_{y_i} \tilde u \partial_{y_j} \tilde u}{2} + h \frac{W(\tilde u)}{\hat \varepsilon}\bigg) \, dz dx,
    \end{align*}
where
    \[|\partial_z \tilde u|^2 := \sum_{i=1,2}|\partial_{z_i} \tilde u|^2, \qquad |\partial_x \tilde u|^2 := \sum_{i=1}^{n-2} |\partial_{x_i} \tilde u|^2.\]
We may assume $\delta_1$ is so small that 
    \[ -\frac{1}{2}|\xi|^2 \leq 4|z|^2hb^{ij}\xi_i \xi_j \leq \frac{1}{2} |\xi|^2 \]
for all $\xi \in \mathbb{R}^n$, so that we obtain 
    \[\int_{\Omega} \hat \varepsilon |\partial_{z} \tilde u|^2 + 4|z|^2 \bigg(\hat \varepsilon|\partial_{x} \tilde u|^2 + h \frac{W(\tilde u)}{\hat \varepsilon} \bigg) \, dz dx \leq 2\int_{\Phi(\Omega)} \hat \varepsilon \frac{|\nabla \hat u|^2}{2} + \frac{W(\hat u)}{\hat \varepsilon} \, dx.\] 
The right-hand side is finite by assumption, so we have
    \[ \int_{\Omega} |\partial_{z} \tilde u|^2 + 4|z|^2 |\partial_{x} \tilde u|^2 \, dz dx < \infty.\]
Using this fact one can extend $\tilde u$ to a function in $W^{1,2}(B_1^2\times B^{n-2})$ such that \eqref{z squared weak} holds for all $\tilde \varphi \in C^1_0(B_1^2\times B^{n-2}_1)$ satisfying $\tilde \varphi(-z,x)=-\tilde \varphi(z,x)$. This is achieved by introducing cutoffs in $C^1_0(\Omega)$ which approximate the characteristic function of $B_1^2\times B^{n-2}_1$ and taking a limit. 

We may now proceed exactly as in \cite[Theorem~3.2]{Struwe} to establish that, for every $m \geq 0$ and partial derivative $\partial^m_y$ of order $m$ in the variables $y$,
    \[\sum_{i=1,2} \|\partial_y^m(\partial_{z_i} \tilde u)\|_{L^2(B^2_{1/2} \times B^{n-2}_{1/2})} \leq C(n, R, C_0, m).\]
By the Sobolev embedding theorem, we conclude that $\partial_{z_i} \tilde u \in C^\infty(B^2_{1/4} \times B^{n-2}_{1/4})$, and we have 
    \begin{equation}\label{z derivatives}
    \sum_{i=1,2} \|\partial_{z_i} \tilde u\|_{C^m(B^2_{1/4} \times B^{n-2}_{1/4})} \leq C(n, R, C_0, m)
    \end{equation}
for a larger constant $C$. It follows that $\tilde u$ is smooth in $z$. Consequently, $\tilde u(z, x) = - \tilde \varphi(-z,x)$ implies that $\tilde u(0, x) = 0$ for each $x \in B^{n-2}_{1/4}$. From this we conclude that $\tilde u$ is smooth also in $x$, since we may write 
    \[\tilde u(z,x) = \int_0^1 \frac{d}{dt} \tilde u(tz, x)\, dx = \int_0^1 z_1\partial_{z_1} \tilde u(tz, x) + z_2 \partial_{z_2} \tilde u(tz, x)\, dt.\]
Moreover, for any $m \geq 0$ and partial derivative $\partial^m_x$ of order $m$ in the variables $x$, we have 
    \begin{equation}\label{x derivatives}
    |\partial^m_x \tilde u(z,x)| \leq |z| \sum_{i=1,2} \|\partial_{z_i} \tilde u\|_{C^m(B^2_{1/4} \times B^{n-2}_{1/4})} \leq C(n, R, C_0, m)|z|. 
    \end{equation}
for $(z,x) \in B^{2}_{1/4} \times B^{n-2}_{1/4}$. Now, the estimates \eqref{z derivatives} and \eqref{x derivatives} can be rewritten as estimates for the original critical section $u$. This gives Proposition~\ref{boundary derivatives}.

\subsection{Energy monotonicity}
We continue working in the setting of \eqref{local}. A crucial step in the proof of Theorem~\ref{main critical} will be to control how the energy of critical sections concentrates in small balls around points in the boundary $\Gamma$. To facilitate this analysis we establish an almost-monotonicity formula for the rescaled energy 
    \[\frac{1}{r^{n-1}} E_\varepsilon(u, B_r(p))\]
at boundary points $p \in \Gamma$. The computations follow \cite{HT}, but various error terms arising from the boundary need to be dealt with. This is achieved using Proposition~\ref{boundary derivatives}. 

We first derive an inner variation identity for vector fields which are tangential on $\Gamma$.

\begin{lemma}\label{first variation}
Let $u \in C^\infty(X,L)$ be a critical section for $E_\varepsilon$. Suppose $\sup |u| < \infty$ and $E_\varepsilon(u) < \infty$, and let $g \in C^1_0(B, \mathbb{R}^n)$ be such that $g|_{\Gamma}$ is tangent to $\Gamma$.  We then have 
    \[\int_X \bigg(\varepsilon \frac{|\nabla u|^2}{2} + \frac{W(u)}{\varepsilon}\bigg)\Div g - \varepsilon \langle \nabla_{Dg} u , \nabla u \rangle  \, dx  = 0.\]
\end{lemma}

\begin{proof}
We define the stress-energy tensor $T$ acting on vector fields $f$ and $h$ by 
    \[T(f,h) := e_\varepsilon(u) f \cdot h - \varepsilon \langle \nabla_f u, \nabla_h u\rangle,\]
where 
    \[e_\varepsilon(u) = \varepsilon \frac{|\nabla u|^2}{2} + \frac{W(u)}{\varepsilon}\]
is the energy density. Using the Euler--Lagrange equation
    \[\varepsilon \Delta u = \frac{1}{\varepsilon} W'(u),\]
we compute 
    \[0 = \Div T = \sum_{i=1}^n (D_i T)(e_i, \cdot),\]
with respect to any orthonormal frame $\{e_i\}$ for $\mathbb{R}^n$. It follows that 
    \[\Div T(g)= \sum_{i =1}^n T(D_i g, e_i ) =: T \cdot Dg,\]
where $T(g)$ denotes the vector field $\sum_{i=1}^n T(g, e_i) e_i$. Integrating this identity over the region $\{\rho > \delta\}$ and applying the divergence theorem yields 
    \begin{align*}
        \int_{\{\rho > \delta\}} T \cdot Dg \, dx = - \int_{\{\rho = \delta\}} T(g) \cdot \grad \rho \, d\mathcal H^{n-1}.
    \end{align*}
Here we assume $\delta < \delta_0$ so that $\{\rho = \delta\}$ is a smooth hypersurface. Since $u$ has finite energy, the dominated convergence theorem implies that the left-hand side converges to  
    \[\int_{X} T \cdot Dg \, dx = \int_X e_\varepsilon(u)\Div g - \varepsilon \langle \nabla_{Dg} u , \nabla u \rangle  \, dx\]
as $\delta \to 0$. Therefore, to prove the claim, it suffices to establish that
    \begin{equation}\label{first var error}
    \int_{\{\rho = \delta\}} T(g) \cdot \grad \rho \, d\mathcal H^{n-1} \to 0
    \end{equation}
as $\delta \to 0$. 

In the following $C$ is a positive constant which does not depend on $\delta$. We have 
    \[T(g) \cdot \grad \rho = e_\varepsilon(u) g \cdot \grad \rho - \varepsilon \langle \nabla_g u, \nabla_{\grad \rho} u\rangle.\]
Since $g|_{\Gamma}$ is tangent to $\Gamma$, in $\{\rho \leq \delta\}$ we may estimate the normal component of $g$ by 
    \[|g^\perp| \leq C \rho.\]
Applying Proposition~\ref{boundary derivatives}, we find that for all sufficiently small $\delta$ we have 
    \begin{align*}
    \varepsilon |\langle \nabla_g u, \nabla_{\grad \rho} u\rangle| &\leq \varepsilon |\nabla_g u| |\nabla_{\grad \rho} u|\\
    &\leq  \varepsilon (|\nabla_{g^\perp} u| + |\nabla_{g^\top} u|)|\nabla_{\grad \rho} u|\\
    &\leq C (1+\varepsilon^{-1})
    \end{align*}
and
    \[e_\varepsilon(u)|g \cdot \grad \rho| \leq  C(1+\varepsilon^{-2} \rho^2 + \varepsilon^{-1}\rho)\]
in $\{\rho \leq \delta\}$. These two estimates imply \eqref{first var error}.
\end{proof}

We are now ready to derive the almost-monotonicity formula for the rescaled energy at boundary points. 

\begin{proposition}\label{energy monotonicity}
Let $u \in C^\infty(X,L)$ be a critical section for $E_\varepsilon$, and suppose $\sup_X |u| < \infty$ and $E_\varepsilon(u, B) < \infty$. Fix a point $p \in \Gamma$ and let $B_r = B_r(p)$. We then have 
    \begin{align*}
        \frac{e^{\Lambda r}}{r^{n-1}} E(u, B_r) - \frac{e^{\Lambda s}}{s^{n-1}} E(u, B_s) &\geq - \int_s^r \frac{e^{\Lambda \tau}}{(1+n^{-1}\Lambda \tau)\tau^n} \int_{B_\tau} \bigg(\varepsilon \frac{|\nabla u|^2}{2} - \frac{W(u)}{\varepsilon} \bigg)\,dx d\tau\\
        &+\varepsilon\int_s^r \frac{1}{(1+n^{-1}\Lambda \tau)\tau^{n-1}} \int_{\partial B_\tau} |\nabla_{\nu} u|^2 \, d\mathcal{H}^{n-1} d\tau
    \end{align*}
for all $s < r < \min\{r_0, \dist(p, \partial B)\}$, where $r_0$ and $\Lambda$ are constants which depend only on $n$ and $\Gamma$, and $\nu$ is the outward unit normal to $\partial B_r$. 
\end{proposition}
\begin{proof}
We may assume $p = 0$. Let $r_0 < \delta_0$ be a small constant to be refined as we proceed. We have that $\rho$ is smooth in $\{\rho < r_0\}$ and the nearest point projection
    \[\pi : \{\rho < r_0\} \to \Gamma\]
is well defined and smooth. We define 
    \[G(x) := \frac{1}{2}\dist_{\Gamma}(\pi(x),0)^2 + \frac{1}{2}\rho(x)^2\]
for $x \in B_{r_0}$, where $\dist_\Gamma$ is the geodesic distance on $\Gamma$. This function is smooth in $B_{r_0}$ provided $r_0$ is smaller than the injectivity radius of $\Gamma$. Note that if $\Gamma$ were affine, we would simply have $G(x) = \frac{1}{2} |x|^2$. Let 
    \[g(x) = \grad G (x).\]
Then $g|_{\Gamma}$ is tangential to $\Gamma$. Moreover, performing a Taylor expansion at the origin shows that we can write $g(x) = x + h(x)$ where $h$ satisfies  
    \begin{equation}\label{h bounds}
    |h(x)| \leq C|x|^2, \qquad |D h(x)| \leq C |x|
    \end{equation}
for $x \in B_{r_0}$, where $C$ is a constant depending only on $n$ and $\Gamma$.

Now suppose $0 < r < \min\{r_0, \dist(p, \partial B)\}$. By Lemma~\ref{first variation}, for any $\varphi \in C^1_0(B_{r})$, we have 
        \[0 = \int_X e_{\varepsilon}(u)\Div (\varphi g) - \varepsilon \langle \nabla_{D(\varphi g)} u , \nabla u \rangle \, dx.\]
Expanding and rearranging gives
    \begin{align*}
        \int_X e_\varepsilon(u) \varphi \Div g - \varepsilon \varphi \langle \nabla_{Dg} u ,\nabla u \rangle \, dx = -\int_X e_\varepsilon(u)\grad \varphi \cdot  g - \varepsilon   \langle \nabla_{g} u , \nabla_{\grad \varphi} u\rangle \,dx.
    \end{align*}
If we allow $\varphi$ to increase to the characteristic function of $B_{r}$ in an appropriate manner, we obtain
    \begin{align*}
        \int_{B_r} e_\varepsilon(u)\Div g - \varepsilon \langle \nabla_{Dg} u ,\nabla u \rangle  \, dx = \int_{\partial B_r} e_\varepsilon(u) (\nu \cdot g) - \varepsilon \langle \nabla_g u, \nabla_\nu u\rangle \, d\mathcal{H}^{n-1}.
    \end{align*}
    
We now insert $g = x + h = r \nu + h$, and so find that
    \begin{align*}
        (n-1)&\int_{B_r} e_\varepsilon(u) - \int_{B_r} \bigg(\varepsilon \frac{|\nabla u|^2}{2} - \frac{W(u)}{\varepsilon} \bigg) + \int_{B_r} \bigg(e_\varepsilon(u)\Div h - \varepsilon \langle \nabla_{Dh} u ,\nabla u \rangle \bigg)\\
        & = r\int_{\partial B_r} e_\varepsilon(u) - \varepsilon r \int_{\partial B_r} |\nabla_\nu u|^2 + \int_{\partial B_r} \bigg(e_\varepsilon(u) (\nu \cdot h) - \varepsilon \langle \nabla_h u, \nabla_\nu u\rangle\bigg).
    \end{align*}
Appealing to \eqref{h bounds} we obtain
    \begin{align*}
        -\int_{B_r} \bigg(e_\varepsilon(u)\Div h - \varepsilon \langle \nabla_{Dh} u ,\nabla u \rangle \bigg) \leq C r\int_{B_r} e_\varepsilon(u)
    \end{align*}
and 
    \begin{align*}
        \int_{\partial B_r} \bigg(e_\varepsilon(u) (\nu \cdot h) - \varepsilon \langle \nabla_h u, \nabla_\nu u\rangle\bigg) \leq Cr^2 \int_{\partial B_r} e_\varepsilon(u),
    \end{align*}
and so conclude that 
    \begin{align*}
        \frac{n-1-C r}{1+Cr}&\int_{B_r} e_\varepsilon(u) - \frac{1}{1+Cr}\int_{B_r} \bigg(\varepsilon \frac{|\nabla u|^2}{2} - \frac{W(u)}{\varepsilon} \bigg) \\
        & \leq r\int_{\partial B_r} e_\varepsilon(u) - \frac{1}{1+Cr}\varepsilon\int_{\partial B_r} |\nabla_\nu u|^2.
    \end{align*}
Let $\Lambda = n C$. Combining the last estimate with
    \begin{align*}
        \frac{d}{dr} \bigg(\frac{e^{\Lambda r}}{r^{n-1}} E(u, B_r)\bigg) &= -(n-1 - \Lambda r) \frac{e^{\Lambda r}}{r^n} E(u,B_r) + \frac{e^{\Lambda r}}{r^{n-1}} \int_{\partial B_r} e_\varepsilon(u),
    \end{align*}
and the inequality
    \[\frac{n-1-Cr}{1+Cr} \geq n-1-\Lambda r,\]
we obtain
    \begin{align*}
        \frac{d}{dr} \bigg(\frac{e^{\Lambda r}}{r^{n-1}} E(u, B_r)\bigg) &\geq - \frac{e^{\Lambda r}}{(1+Cr)r^n} \int_{B_r} \bigg(\varepsilon \frac{|\nabla u|^2}{2} - \frac{W(u)}{\varepsilon} \bigg)\\
        &+ \frac{e^{\Lambda r}}{(1+Cr)r^{n-1}} \varepsilon  \int_{\partial B_r} |\nabla_\nu u|^2.
    \end{align*}
Integrating from $s$ to $r$ now gives the claim.
\end{proof}

\subsection{Interior estimates} We will eventually show that, in the setting of \eqref{sequence of sections}, the almost-monotonicity formula established in Proposition~\ref{energy monotonicity} improves, so that the rescaled energy becomes genuinely monotone as $\varepsilon \to 0$. This requires showing that the terms on the right-hand of the almost-monotonicity formula become nonnegative as $\varepsilon \to 0$. To this end, we use the maximum principle to establish supremum and gradient estimates for solutions of the scalar Allen--Cahn equation defined in a ball in $\mathbb{R}^n$. Up to rescaling and choosing a unit section, these apply to critical sections of $L$ in any ball $B' \subset X$. Our first estimate is slightly stronger than \cite[Proposition~3.2]{HT}. 

\begin{lemma}\label{interior sup}
Let $u$ be a $C^2$ function defined on $\bar B_{r+kR}$, where $k \geq 1$ is an integer and $R \geq 2$. Suppose $u$ solves the scalar Allen--Cahn equation $\Delta u = W'(u)$, and that $|u| \leq C_0$. We then have 
    \[\sup_{B_{r}} |u| \leq 1 + 2^{1+k}C_0R^{-2k}.\]
\end{lemma}
\begin{proof}
Let $\zeta$ be a smooth function on $\mathbb{R}^n$ such that $\zeta = 2C_0R^{-2}$ in $B_{r+(k-1)R}$ and $\zeta \geq C_0$ on $\partial B_{r+kR}$. We may assume $2C_0R^{-2} \leq \zeta \leq C_0$ and $|\Delta \zeta| \leq 2C_0R^{-2}$. With the aim of deriving a contradiction, suppose
    \[\sup_{B_{r+(k-1)R}} u \geq 1 + 4C_0R^{-2}.\]
Let $\psi := u - \zeta - 1$. Since $\psi$ is negative on $\partial B_{r+kR}$, we have that $\psi$ attains a maximum of at least $2C_0R^{-2}$ at some point $x_0 \in B_{r+kR}$. At the point $x_0$,
    \begin{align*}
        0 &\geq \Delta \psi\\
        &=W'(u) - \Delta \zeta\\
        &\geq \psi \int_0^1 W''(tu+(1-t)(\zeta+1))\,dt + W'(\zeta+1) - 2C_0R^{-2}\\
        &> 2\psi - 2C_0R^{-2},
    \end{align*}
where $W''(u) = 3u^2 - 1$. This is a contradiction, so in fact
    \[\sup_{B_{r+(k-1)R}} u \leq 1+4C_0R^{-2}.\]

We now repeat the same argument, but with $\zeta$ chosen so that $\zeta = C_0$ on $\partial B_{r+(k-1)R}$, $\zeta = 2C_0R^{-2}$ in $B_{r+(k-2)R}$, and $|\Delta \zeta| \leq 2C_0R^{-2}$. If $\sup_{B_{r+(k-2)R}} u \geq 1+8C_0R^{-4}$, we obtain a contradiction by applying the maximum principle to $\psi:=u - 4R^{-2}\zeta - 1$, exactly as before. Continuing in this way, we obtain the desired upper bound for $u$ after $k$ iterations. The lower bound is analogous.
\end{proof}

The next estimate slightly improves \cite[Proposition~3.3]{HT}.

\begin{proposition}\label{discrepancy upper}
Let $u \in C^2(\bar B_{r+3R})$ be a solution of $\Delta u = W'(u)$ such that $|u| \leq C_0$. Define
    \[\xi := \frac{|\nabla u|^2}{2} - W(u), \qquad \xi_+ := \max\{\xi, 0\}.\]
There are constants $R_0$ and $C$ depending only on $n$ and $C_0$ such that if $R \geq R_0$ then
    \[\sup_{B_{r}} \xi \leq C \max\bigg\{ R^{-2/3} \sup_{\partial B_{r+R}} \xi_+^{2/3}, R^{-2}\bigg\}.\]
Moreover, for every $\delta > 0$ there is an integer $k_0 = k_0(\delta)$ with the following property. If $u \in C^\infty(\bar B_{r+(k_0 + 2)R})$ is a solution of $\Delta u = W'(u)$ such that $|u| \leq C_0$, and $R \geq R_0$, then we have 
    \[\sup_{B_r} \xi \leq C R^{-2+\delta},\]
where $C = C(n,C_0,\delta)$. 
\end{proposition}
\begin{proof}
First observe that for $R \geq 2$, Lemma~\ref{interior sup} implies
    \[\sup_{B_{r+R}}|u|\leq1+8C_0R^{-4}.\]

Let us define
    \[\xi_G := \frac{|\nabla u|^2}{2} - W(u) - G(u),\]
where $G:\mathbb{R}\to\mathbb{R}$ is a smooth function to be chosen later. At any point where $|\nabla u| > 0$, we have 
    \begin{align*}
        \Delta \xi_G - 2&\frac{(W'(u) + G'(u))}{|\nabla u|^2} \nabla \xi_G \cdot \nabla u + 2G''(u)\xi_G\\
            &\geq G'(u)^2 + G'(u)W'(u) - 2G''(u)(W(u) + G(u)).
    \end{align*}
We set
    \[\mu := \max \bigg\{\max_{\partial B_{r+R}} \xi_+, \frac{125}{R^2}\bigg\}\]
and
    \[G(u) := \lambda (1-u^2/Q), \qquad \lambda := Q\bigg(\frac{\mu}{R}\bigg)^{2/3},\]
where $Q \geq 8$ will be chosen later. Let $\zeta \in C^\infty(\bar B_{r+R})$ be such that 
    \[\zeta = 1 \text{ in } B_{r}, \;\;\;\; 0 \leq \zeta \leq 1 \text{ in } B_{r+R}, \;\;\;\; |\nabla \zeta| \leq 2R^{-1}, \;\;\;\; |\Delta \zeta| \leq 2R^{-2}.\]
Define $\tilde \xi := \xi_G + \mu \zeta$. 

With the aim of deriving a contradiction, suppose $\sup_{B_r} \xi_G \geq \lambda$. It follows that $\sup_{B_r} \tilde \xi \geq \lambda + \mu$. On the other hand, $\tilde \xi \leq \mu$ on $\partial B_{r+R}$, so we conclude that $\tilde \xi$ attains an interior mamximum at some point $x_0 \in B_{r+R}$. At the point $x_0$ we have 
    \[|\nabla u| \geq \lambda^{1/2}, \qquad |\nabla \xi_G| = \mu|\nabla \zeta| \leq 2\mu R^{-1}, \qquad \Delta \xi_G = - \mu \Delta \zeta \leq 2\mu R^{-2}.\]
Therefore, at $x_0$, 
    \begin{align*}
        \Delta \xi_G - 2&\frac{(W'(u) + G'(u))}{|\nabla u|^2} \nabla \xi_G \cdot \nabla u + 2G''(u)\xi_G\\
            &\leq 2 \mu R^{-2} + 4 \mu |W'(u)| \lambda^{-1/2}R^{-1} + 8 \mu Q^{-1} \lambda^{1/2}R^{-1},
    \end{align*}
and
    \begin{align*}
        G'(u)^2 + G'(u)&W'(u) - 2G''(u)(W(u) + G(u)) \\
        &\geq 4Q^{-2} \lambda^{2}u^2 - 2Q^{-1}\lambda u W'(u) + 4Q^{-1}\lambda W(u).
    \end{align*}
Consequently, at the point $x_0$,
        \begin{align}\label{discrepancy max}
        2 \mu R^{-2} + 4 \mu & |W'(u)| \lambda^{-1/2}R^{-1} + 8 \mu Q^{-1} \lambda^{1/2}R^{-1} \notag \\
        &\geq 4Q^{-2} \lambda^{2}u^2 - 2Q^{-1}\lambda u W'(u) + 4Q^{-1}\lambda W(u).
    \end{align}
We consider three cases. 

\textit{Case 1.} If $|u(x_0)| \in [0,1/2]$ then \eqref{discrepancy max} gives 
        \[2 \mu R^{-2} + 4 \mu \alpha \lambda^{-1/2}R^{-1} + 8 \mu Q^{-1} \lambda^{1/2}R^{-1} \geq 2Q^{-1}\lambda W(u) \geq 4\beta Q^{-1} \lambda,\]
where
    \[\alpha := \max_{s \in [0,1/2]}|W'(s)|, \qquad \beta := W(1/2).\]
Inserting $\lambda = Q(\mu/ R)^{2/3}$, we obtain
        \[2 \mu R^{-2} + 4 \alpha \mu^{2/3} Q^{-1/2} R^{-2/3} + 8 \mu^{4/3} Q^{-1/2} R^{-4/3} \geq 4\beta \mu^{2/3} R^{-2/3},\]
and hence
    \begin{equation*}
        2\mu^{1/3} R^{-4/3} + 4 \alpha Q^{-1/2} + 8 Q^{-1/2}\mu^{4/3} R^{-2/3} \geq 4\beta.
    \end{equation*}
Since $u$ is smooth and bounded in $B_{r+3R}$, standard elliptic estimates imply that $\mu$ can be bounded from above in terms of $n$ and $C_0$, so we obtain a contradiction provided that $Q$ and $R$ are sufficiently large depending on $n$, $C_0$, $\alpha$ and $\beta$.
    
\textit{Case 2.} If instead $|u(x_0)| \in [1/2,1]$, we estimate 
    \begin{align*}
        4Q^{-2} \lambda^{2}u^2 - 2Q^{-1}\lambda u W'(u) + 4Q^{-1}\lambda W(u) \geq Q^{-2} \lambda^{2} + Q^{-1}\lambda |W'(u)|,
    \end{align*}
and so obtain
    \[2 \mu R^{-2} +  |W'(u)| (4 \mu\lambda^{-1/2}R^{-1} - Q^{-1}\lambda) + 8 \mu Q^{-1} \lambda^{1/2}R^{-1} \geq Q^{-2} \lambda^{2}\]
from \eqref{discrepancy max}. Inserting $\lambda = Q (\mu/R)^{2/3}$ yields 
    \[2\mu R^{-2} + |W'(u)| (4Q^{-1/2} - 1)\mu^{2/3}R^{-2/3} + 8 Q^{-1/2} \mu^{4/3} R^{-4/3} \geq \mu^{4/3} R^{-4/3}.\]
For $Q \geq 16^2$ we obtain
    \[4\mu R^{-2} \geq \mu^{4/3}R^{-4/3}\]
and hence
    \[4 R^{-2/3} \geq \mu^{1/3}.\]
Recall however that $\mu \geq 125/R^2$. We have therefore arrived at a contradiction.

\textit{Case 3.} Finally, we consider the case that $|u(x_0)| \geq 1$. Recall we used Lemma~\ref{interior sup} to conclude that $|u(x_0)| \leq 1+8C_0R^{-4}$. At $x_0$ we have 
    \[2 \mu R^{-2} + 4 \mu W'(u) \lambda^{-1/2}R^{-1} + 8 \mu Q^{-1} \lambda^{1/2}R^{-1} \geq 4Q^{-2} \lambda^{2} - 2Q^{-1}\lambda u W'(u).\]
Since $|u| \in [1, 1+8C_0R^{-4}]$, provided $R$ is large enough so that $8C_0R^{-4} \leq 1$ we have 
    \[W'(u) \leq 48C_0R^{-4}, \qquad -uW'(u)\geq-480C_0R^{-4},\]
and hence 
    \[2 \mu R^{-2} + 184 C_0 \mu \lambda^{-1/2}R^{-5} + 8 \mu Q^{-1} \lambda^{1/2}R^{-1} \geq 4Q^{-2} \lambda^{2} - 960 C_0 Q^{-1}\lambda R^{-4}.\]
Inserting $\lambda = Q\mu^{2/3}R^{-2/3}$ we obtain
    \[2 \mu R^{-2} + (184 Q^{-1/2} + 960) C_0 \mu^{2/3} R^{-4-2/3}  \geq (4 - 8 Q^{-1/2})\mu^{4/3}R^{-4/3}.\]
For $Q \geq 16^2$, since $\mu \geq 125R^{-2}$,
    \[(4 - 8 Q^{-1/2})\mu^{4/3}R^{-4/3} \geq \mu^{1/3} R^{-4/3} \mu  + \mu^{2/3}R^{-4/3} \mu^{2/3} \geq 5 \mu R^{-2} + 25 \mu^{2/3}R^{-2-2/3},\]
and hence 
    \[2 \mu R^{-2} + (184 Q^{-1/2} + 960) C_0 \mu^{2/3} R^{-4-2/3} \geq 5 \mu R^{-2} + 25 \mu^{2/3}R^{-2-2/3}.\]
This is a contradiction provided $R$ is sufficiently large depending on $C_0$.

Combining the three cases, we obtain $\sup_{B_r} \xi_G \leq \lambda$. This in turn gives
    \[\sup_{B_{r}} \xi \leq C \max\bigg\{ R^{-2/3} \sup_{\partial B_{r+R}} \xi_+^{2/3}, R^{-2}\bigg\},\]
with $C := 50Q$.

Now suppose $u$ is defined in $\bar B_{r+(k+2)R}$ and satisfies $|u| \leq C_0$. The estimate just derived (applied with $r + (k-l)R$ in place of $r$) tells us that
    \[\sup_{B_{r+(k-l)R}} \xi \leq C \max\bigg\{ R^{-2/3} \sup_{\partial B_{r+(k-l+1)R}} \xi_+^{2/3}, R^{-2}\bigg\}.\]
Applying this estimate iteratively, we obtain
    \[\sup_{B_{r+(k-l)R}} \xi \leq C_l R^{-p_l}\]
for each $l \leq k$, where $p_l$ satisfies the relation $p_l = \tfrac{2}{3}(1+p_{l-1})$. It follows that as $l\to\infty$ we have $p_l \to 2$. So if $k$ is sufficiently large depending on $\delta$, then $p_k \geq 2 - \delta$, and hence
    \[\sup_{B_r} \xi \leq C R^{-2+\delta},\]
where $C = C(n,C_0,\delta)$. This completes the proof. 
\end{proof}

\subsection{$L^1$-convergence of the discrepancy} We assume \eqref{local} and \eqref{sequence of sections}. We write $\xi_k$ for the discrepancy of $u_k$, namely 
   \[\xi_k := \varepsilon_k \frac{|\nabla u_k|^2}{2} - \frac{W(u_k)}{\varepsilon_k}\]
The main result we want to prove is the following (cf.~\cite[Proposition~4.3]{HT}).

\begin{proposition}\label{discrepancy integral}
As $k \to \infty$ the discrepancy $\xi_k \to 0$ in $L^1_{\loc}(B)$. 
\end{proposition}

We need different arguments for the positive and negative parts of $\xi_k$. We first consider the positive part, $\xi_{k,+} := \max\{\xi_k, 0\}$. 

\begin{lemma}\label{discrepancy integral upper}
As $k \to \infty$ we have $\xi_{k,+} \to 0$ in $C^0_{\loc}(X)$ and $L^1_{\loc}(B)$.
\end{lemma}
\begin{proof}
Fix a constant $\delta \in (0,1)$. Suppose $x \in X$ is such that $\rho(x) > R \varepsilon_k$ and $\dist(x,\partial B) > R\varepsilon_k$, where 
    \[R := (1+ (k_0+2) R_0),\]
and $k_0=k_0(\delta)$ and $R_0=R_0(n,C_0)$ are the constants referred to in Proposition~\ref{discrepancy upper}. We then have
    \[B_{\varepsilon_k R}(x) \subset X.\]
After rescaling and applying Proposition~\ref{discrepancy upper}, we find that 
    \[\sup_{B_{\varepsilon_k }(x)} \xi_{k,+} \leq C_2(n,C_0,\delta) \varepsilon_k^{1-\delta}.\]
Since every compact subset of $X$ is in the set $\{\rho > \varepsilon_k R\} \cap \{\dist(\cdot, \partial B) > R\varepsilon_k\}$ for sufficiently $k$, we conclude that $\xi_{k,+} \to 0$ in $C^0_{\loc}(X)$.

In the set $\{\rho \leq \varepsilon_k R\}\cap\{\dist(\cdot, \partial B) > 2R\varepsilon_k\}$ we can apply the boundary derivative estimates established in Proposition~\ref{boundary derivatives}. These provide a constant $C$ which is independent of $k$ such that
    \[|\nabla u_k|^2 \leq C(\varepsilon_k^{-1} \rho^{-1} + \varepsilon_k^{-3}\rho),\]
and hence
    \[\xi_{k,+} \leq C( \rho^{-1} + \varepsilon_k^{-2}\rho)\]
in $\{\rho \leq \varepsilon_k R\}\cap\{\dist(\cdot, \partial B) > 2R\varepsilon_k\}$. By the coarea formula, for large $k$,
    \[\int_{\{\rho \leq \varepsilon_k R\} \cap \{\dist(\cdot, \partial B) > 2R\varepsilon_k\} } ( \rho^{-1} + \varepsilon_k^{-2}\rho) \leq C \int_0^{\varepsilon_k R}( 1 + \varepsilon_k^{-2}\rho^2) \,d\rho \leq C \varepsilon_k R^3,\]
where we have increased $C$ as necessary. For any compact set $K$ in $B$ we have
    \begin{align*}\int_K \xi_{k,+} &= \int_{K \cap \{\rho\leq \varepsilon_k R\} \cap \{\dist(\cdot, \partial B) > 2R\varepsilon_k\}} \xi_{k,+} + \int_{K \cap \{\rho > \varepsilon_k R\} \cap \{\dist(\cdot, \partial B) > 2R\varepsilon_k\}} \xi_{k,+} \\
    &\leq C \varepsilon_k  R^3 + C_2 \mathcal H^{n}(K) \varepsilon_k^{1-\delta}\end{align*}
for all large $k$. In particular, $\xi_{k,+} \to 0$ in $L^1_{\loc}(B)$ as $\varepsilon_k \to 0$. 
\end{proof}

We now turn to the negative part of $\xi_{k}$. That this quantity converges to 0 in $L^1_{\loc}(B)$ is established below in Lemma~\ref{discrepancy integral lower}. For the proof we need some notation and further lemmas. 

For our fixed sequence of critical sections $u_k$, let us define a sequence of energy measures $\mu_k$ by
    \[\mu_k = \bigg(\varepsilon_k \frac{|\nabla u_k|^2}{2} + \frac{W(u_k)}{\varepsilon_k}\bigg) \, dx.\]
Then $\mu_k(A) = E_{\varepsilon_k}(u_k, A)$ for each Borel set $A \subset B$. The energy bound $E_{\varepsilon_k}(u_k) \leq C_1$ implies that $\mu_k$ is a sequence of uniformly bounded Radon measures, so after passing to a subsequence, we can assume that $\mu_k$ weak*-converges to some limiting Radon measure $\mu$. That is, 
    \[\lim_{k\to\infty} \int \varphi \, d\mu_k = \int \varphi \, d\mu\]
for every $\varphi \in C_0(B)$.

Consider a loop $\gamma$ in $X$ which satisfies $\operatorname{link}_{\mathbb{Z}}(\gamma, \Gamma) = \pm 1$. Then every smooth section of $L$ must vanish at at least one point on $\gamma$. This has the following consequence for the limiting energy measure $\mu$. 

\begin{lemma}\label{support on loops} Suppose $\mu_k \to \mu$ in the weak* sense. Let $\gamma$ be a smooth loop in $X$ such that $\operatorname{link}_{\mathbb{Z}}(\gamma, \Gamma) = \pm 1$. There is then a point $p \in \gamma$ such that 
    \[\frac{1}{r^{n-1}}\mu(B_r(p)) \geq c\]
for all $r < \dist(\gamma, \Gamma \cup \partial B)/2$, where $c = c(n,C_0)$. In particular, $p \in \supp \mu$.
\end{lemma}
\begin{proof}
Suppose $r < \dist(\gamma, \Gamma \cap \partial B)/2$, so that $\bar B_{r}(x)$ is contained in $X$ for every $x \in \gamma$. We know that $u_k$ vanishes at some point $p_k$ in $\gamma$ for every $k$. After passing to a subsequence, we may assume $p_k \to p \in \gamma$. We have 
    \[\frac{2^{n-1}}{r^{n-1}}\mu(B_r(p)) \geq \limsup_{k\to \infty}\frac{2^{n-1}}{r^{n-1}}\mu_k(B_{r/2}(p_k)),\]
so the interior energy monotonicity formula in \cite{HT} and Proposition~\ref{discrepancy integral upper} imply    
    \[\frac{2^{n-1}}{r^{n-1}}\mu(B_r(p)) \geq \liminf_{k\to \infty} \frac{1}{\varepsilon_k^{n-1}}E_{\varepsilon_k}(u_k, B_{\varepsilon_k}(p_k)).\]
After rescaling and applying standard elliptic estimates, we find that there is a positive constant $c$ depending only on $n$ and $C_0$ such that 
    \[\frac{1}{\varepsilon_k^{n-1}} E(u_k, B_{\varepsilon_k}(p_k)) \geq c\]
for all sufficiently large $k$. Consequently,     
    \[\frac{1}{r^{n-1}}\mu(B_r(p)) \geq \frac{c}{2^{n-1}}.\]
\end{proof}

We note that the conclusion of Lemma~\ref{support on loops}, together with the fact that $\supp \mu$ is closed, implies that $\Gamma \cap B \subset \supp \mu$. To see this, approximate any point in $\Gamma \cap B$ by a sequence of contracting loops and apply the lemma. 

Using Lemma~\ref{support on loops}, we can now prove density bounds for $\mu$, as in \cite[Proposition~4.1]{HT}. 

\begin{lemma}\label{local energy bounds}
Given $p \in \supp \mu$, there are positive constants $r_1 = r_1(n,\Gamma)$ and $D = D(n, \Gamma, C_0, C_1)$ such that 
    \begin{equation}\label{limit measure bounds}
        D^{-1} \leq \frac{1}{r^{n-1}}\mu(B_r(p)) \leq D\bigg(1 + \frac{1}{\dist(p,\partial B)^{n-1}}\bigg)
    \end{equation}
for all $r < \min\{r_1, \dist(p, \partial B)/2\}$. 
\end{lemma}

\begin{proof}
Let $D$ be a large positive constant, which may depend only on $n$, $\Gamma$, $C_0$ and $C_1$, and may increase as we proceed. We first prove the lower bound in \eqref{limit measure bounds}. We consider two cases. 

\emph{Case 1.} Suppose first that $p \in \Gamma$. Let $\gamma$ denote the set of points in the $2$-plane normal to $\Gamma$ at $p$ which lie at distance $r/2$ from $p$. There is a positive constant $r_1 = r_1(n,\Gamma)$ such that, for $0 < r < \min\{r_1, \dist(p, \partial B)\}$,
    \begin{itemize}
        \item $\gamma$ is smooth loop in $X$,
        \item $\operatorname{link}_{\mathbb Z}(\gamma, \Gamma) = 1$,
        \item and $B_{r/4}(x) \subset B_r(p)$ for each $x \in \gamma$. 
    \end{itemize}
By Lemma~\ref{support on loops}, there exists a point $q \in \gamma$ such that 
    \[\frac{4^{n-1}}{r^{n-1}}\mu(B_{r/4}(q)) \geq c(n,C_0).\]
Therefore,
    \begin{equation}\label{density lower bdy}
    \frac{1}{r^{n-1}}\mu(B_r(p)) \geq \frac{1}{r^{n-1}} \mu(B_{r/4}(q)) \geq \frac{c}{4^{n-1}}
    \end{equation}
for $0 < r < \min\{r_1, \dist(p, \partial B)\}$.

\emph{Case 2.} Next, suppose $p \in \supp \mu \setminus \Gamma$. Consider a constant $0 < r < \min\{r_1, \dist(p, \partial B)\}$. If $B_{r/2}(p)$ intersects $\Gamma$, then there is a point $q \in \Gamma$ such that $B_{r/2}(q) \subset B_{r}(p)$, and hence by \eqref{density lower bdy} we have
    \[\frac{1}{r^{n-1}}\mu(B_r(p)) \geq \frac{1}{r^{n-1}}\mu(B_{r/2}(q)) \geq  \frac{c}{8^{n-1}}.\]
If instead $B_{r/2}(p)$ is disjoint from $\Gamma$ then we argue differently. Since $p \in \supp \mu$, there is a sequence $p_k \to p$ such that $|u_k(p_k)| \geq 3/4$ for all sufficiently large $k$ (see e.g. the proof of \cite[Proposition~4.1]{HT}). Moreover, by the interior almost-monotonicity formula and Proposition~\ref{discrepancy integral upper}, we have 
    \[\frac{2^{n-1}}{r^{n-1}}\mu(B_{r/2}(p)) \geq \liminf_{k\to\infty} \frac{1}{\varepsilon_k^{n-1}} E_{\varepsilon_k}(u_k, B_{\varepsilon_k}(p_k)).\]
After rescaling by $\varepsilon_k$ and using standard elliptic estimates, we find that the right-hand side of the last inequality is bounded from below by a positive constant $c = c(n,C_0)$. Therefore, 
    \[\frac{1}{r^{n-1}}\mu(B_r(p)) \geq \frac{1}{r^{n-1}}\mu(B_{r/2}(p)) \geq \frac{c}{2^{n-1}}.\]

Combining the two cases, we obtain 
    \[\frac{1}{r^{n-1}}\mu(B_r(p)) \geq D^{-1} \]
for all $p \in \supp \mu$ and $0 < r < \min\{r_1, \dist(p, \partial B)\}$. 

We now turn to the upper bound in \eqref{limit measure bounds}. Once again, we consider two cases. 

\emph{Case 1.} Suppose first that $p \in \Gamma$. By Proposition~\ref{energy monotonicity}, provided $r_0$ is sufficiently small depending on $n$ and $\Gamma$, we have 
    \begin{align*}
        \frac{e^{\Lambda r}}{r^{n-1}} &E_{\varepsilon_k}(u_k, B_r(p)) - \frac{e^{\Lambda s}}{s^{n-1}} E_{\varepsilon_k}(u_k, B_s(p)) \\
        &\geq - \int_s^r \frac{e^{\Lambda \tau}}{(1+n^{-1}\Lambda\tau)} \frac{1}{\tau^n} \int_{B_\tau(p)} \bigg(\varepsilon_k \frac{|\nabla u_k|^2}{2} - \frac{W(u_k)}{\varepsilon} \bigg)_+\,dx d\tau
    \end{align*}
for all $s < r < \min\{r_0, \dist(p,\partial B)\} =: d$. The inner integrand on the right is $\xi_{k,+}$, which tends to 0 in $L^1_{\loc}(B)$ as $k \to \infty$ by Lemma~\ref{discrepancy integral upper}. Therefore, after passing to the limit, we find that 
    \[\frac{e^{\Lambda r}}{r^{n-1}} \mu(B_r(p)) \leq \frac{e^{\Lambda d}}{d^{n-1}} \mu(B_{d}(p)) \leq \frac{e^{\Lambda d}}{d^{n-1}} C_1\]
for $0 < r < d$. In particular, 
    \begin{equation}\label{density upper}
        \frac{1}{r^{n-1}}\mu(B_r(p)) \leq D\bigg(1+\frac{1}{\dist(p, \partial B)^{n-1}}\bigg)
    \end{equation}
holds for every $0 < r < d$. 

\emph{Case 2.} Suppose next that $p \in \supp \mu \setminus \Gamma$, and let $d = \min\{r_0, \dist(p,\partial B)/2\}$. If there is a point $q \in B_{d/2}(p) \cap \Gamma$ then for $0 < r < d/2$, using \eqref{density upper} we obtain
    \[\frac{1}{r^{n-1}} \mu(B_{r}(p)) \leq \frac{1}{r^{n-1}} \mu(B_d(q)) \leq D \bigg(1+\frac{1}{\dist(q, \partial B)^{n-1}}\bigg).\]
It is easy to check that $\dist(p, \partial B) \leq \tfrac{4}{3} \dist(q,\partial B)$, and hence
    \[\frac{1}{r^{n-1}} \mu(B_{r}(p)) \leq D \bigg(1+\frac{1}{\dist(p, \partial B)^{n-1}}\bigg)\]
for $0 < r < d/2$. If instead $B_{d/2}(p)$ is disjoint from $\Gamma$, using the interior almost-monotonicity formula and Proposition~\ref{discrepancy integral upper} we obtain
    \[\frac{1}{r^{n-1}}\mu(B_r(p)) \leq \frac{2^{n-1}}{d^{n-1}}\mu(B_{d/2}(p)) \leq D\bigg(1+\frac{1}{\dist(p, \partial B)^{n-1}}\bigg)\]
for $0 < r < d/2$. 

Combining the two cases yields the upper bound in \eqref{limit measure bounds}. Note that we can decrease $r_1$ if necessary to ensure $r_1 \leq r_0/2$.
\end{proof}

We may now conclude that the negative part of the discrepancy decays to $0$ in $L^1_{\loc}$. Combining this result with Lemma~\ref{discrepancy integral upper} gives Proposition~\ref{discrepancy integral}. 

\begin{lemma}\label{discrepancy integral lower}
The negative part of the discrepancy, $\xi_{k,-} := \max\{-\xi_k, 0\}$, decays to $0$ in $L^1_{\loc}(B)$ as $k \to \infty$.
\end{lemma}

\begin{proof}
The proof follows \cite[Proposition~4.3]{HT}. Let us define a Radon measure on $B$ by
    \[\xi_{k,-}(A) := \int_A \xi_{k,-}.\]
Suppose the claim is false, so that there is a Borel set $A \subset B$ for which $\xi_{k,-}(A)$ does not converge to 0. After passing to a subsequence, we may assume that $\liminf_{k\to\infty}\xi_{k,-}(A)>0$. We may also assume that $\xi_{k,-}$ and $\mu_k$ weak*-converge to Radon measures $\xi_-$ and $\mu$.

We first claim that, for each $p \in A \cap \supp \mu$, the quantity 
    \[\delta := \liminf_{r\to\infty} \frac{\xi_{-}(B_r(p))}{\mu(B_r(p))}\]
is zero. Suppose to the contrary that there is a $p \in A \cap \supp \mu$ for which $\delta > 0$. Then, for all sufficiently small values of $r$, we have
    \[\xi_-(B_r(p)) > \frac{\delta}{2} \mu(B_r(p)),\]
and hence
    \[\xi_{-} (B_r(p)) > \frac{\delta}{2 D} r^{n-1},\]
where $D$ is independent of $k$, by Lemma~\ref{local energy bounds}. We derive a contradiction using energy monotonicity. If $p \in \Gamma$ then we use the almost-monotonicity formula derived in Proposition~\ref{energy monotonicity}, whereas if $p \in X$ we use the corresponding interior formula from \cite{HT}. The argument is essentially the same in both cases, so we only describe the boundary case in detail. For all $k$ (and sufficiently small $s < r$) we have 
    \begin{align*}
        \frac{e^{\Lambda r}}{r^{n-1}} &E_{\varepsilon_k}(u_k, B_r(p)) \geq - \int_s^{r/2} \frac{e^{\Lambda \tau}}{(1+n^{-1}\Lambda\tau)} \frac{1}{\tau^n} \int_{B_\tau(p)} \xi_{k}\, d\tau
    \end{align*}
Sending $k$ to infinity, and using the fact that the positive part of the discrepancy goes to 0 in $L^1_{\loc}(B)$ (by Lemma~\ref{discrepancy integral upper}), we obtain 
    \[\frac{e^{\Lambda r}}{r^{n-1}}C_1 \geq \int_s^{r/2} \frac{e^{\Lambda \tau}}{(1+n^{-1}\Lambda \tau)}  \frac{\xi_{-}(B_\tau(p))}{\tau^n} \, d\tau \geq \frac{\delta}{2D} \int_s^{r/2} \frac{e^{\Lambda \tau}}{(1+n^{-1}\Lambda\tau)} \frac{1}{\tau} \, d\tau\]
for sufficiently small $s < r$. The right-hand side becomes unbounded as $s \to 0$, so this is a contradiction. That is, $\delta = 0$ for every $p \in A \cap \supp \mu$. 

By a standard result in measure theory (see eg. Lemma~1.2 on p. 47 of Evans--Gariepy), we conclude that
    \[\xi_-(A \cap \supp \mu) = 0.\]
But the inequality $\xi_{k,-}(A) \leq \mu_k(A)$ implies $\supp \xi_- \subset \supp \mu$, so in fact $\xi_-(A) = 0$, contrary to our initial assumption. 
\end{proof} 


\section{Stationary varifolds from critical sections}\label{varifolds}

Given an open subset $U \subset \mathbb{R}^{n}$, we write $G(U)$ for the Grassmannian bundle of unoriented $(n-1)$-planes over $U$. Each point in $G(U)$ is of the form $(x,S)$, with $x \in U$ and $S \in T_x U$.

An $(n-1)$-varifold on $U$ is a nonnegative Radon measure on $G(U)$. A sequence of varifolds $V_k$ weak*-converges to a varifold $V$ if for every $\varphi \in C_0(G(U))$ we have 
    \[\int_{U} \varphi(x,S)\, dV_k(x,S) \to \int_U \varphi(x,S) \, dV(x,S).\]
Given a varifold $V$, we define its mass $\|V\|$ to be the Radon measure on $U$ such that 
    \[\int_U \varphi(x) \, d\|V\|(x) := \int_{G(U)} \varphi(x) \, dV(x,S)\]
for all $\varphi \in C_0(U)$. Given a vector field $g \in C^1_0(U, \mathbb{R}^n)$, the first variation of $V$ with respect to $g$ is 
    \[\delta V (g) = \int Dg(x) \cdot S \, dV(x,S),\]
where 
    \[Dg(x) \cdot S := \sum_{i=1}^{n-1}  D_{e_i} g(x) \cdot e_i\]
and $\{e_i\}_{i=1}^{n-1}$ is any orthonormal basis for $S$.

Given a smooth hypersurface $\Sigma$ in $U$, we may define a measure on $G(\mathbb{R}^n)$ by the requirement that
    \begin{equation} \label{smooth varifold} V_{\Sigma}(A) := \mathcal H^{n-1}(\{x \in \Sigma : (x, T_x\Sigma) \in A\})
    \end{equation}
for every open $A$. If $\mathcal H^{n-1}(\Sigma, K) < \infty$ for every compact $K \subset \mathbb{R}^n$, then $V_\Sigma$ is an $(n-1)$-varifold on $\mathbb{R}^n$. 

\subsection{The associated varifolds} We assume either that $M = \mathbb{R}^n$ or $M = B$, and write $X = M \setminus \Gamma$, where $\Gamma$ is a compact codimension-two submanifold of $\mathbb{R}^n$. Let $L \to X$ be some real line bundle, and suppose $u$ is a smooth section of $L$ with $E_\varepsilon(u) < \infty$. We then define $w : X \to \mathbb{R}$ by
    \[w := \Phi \circ |u|, \qquad \Phi(t) := \int_0^t \sqrt{\frac{(1-s^2)^2}{8}} \,ds.\]
Let $\sigma := 2 \Phi(1)$, and observe that $2\sigma$ is precisely the Allen--Cahn energy of the 1-dimensional heteroclinic solution $z \mapsto \tanh(\tfrac{z}{\varepsilon\sqrt{2}})$. Sard's theorem implies that $\Sigma_t := w^{-1}(t)$ is a smooth hypersurface in $X$ for almost every $t \in \mathbb{R}$. Moreover, by the coarea formula, 
    \[\int_{\mathbb{R}} \mathcal H^{n-1}(\Sigma_t) \, dt = \int_X |d w| \, dx = \int_{X} \sqrt{\frac{W(u)}{2}}|\nabla u| \leq \frac{1}{2} E_\varepsilon(u).\]
We conclude that $\Sigma_t$ is smooth and satisfies $\mathcal H^{n-1}(\Sigma_t) < \infty$ for almost every $t \in \mathbb{R}$. For each such $t$ we define a varifold $V_{\Sigma_t}$ on $M$ as in \eqref{smooth varifold}. We then define a measure on $G(M)$ by setting
    \[V(A) := \frac{1}{\sigma}\int_{\mathbb R} V_{\Sigma_t}(A) \, dt\]
for the open sets $A \subset G(M)$. Note that, by definition, $\|V\|(\Gamma \cap B) = 0$. Therefore, by the computation using the coarea formula above, we have 
    \[\|V\|(M) = \|V\|(X) = \frac{1}{\sigma}\int_{X} \sqrt{\frac{W(u)}{2}}|\nabla u| \leq \frac{1}{2\sigma} E_\varepsilon(u).\]
Consequently, $\|V\|(M) < \infty$, and so $V$ is a varifold on $M$. We refer to $V$ as the varifold associated with $u$. 

For a vector field $g \in C^1_0(M, \mathbb{R}^n)$ we have 
    \[\delta V (g) = \int_{G(M)} Dg(x) \cdot S \, dV(x,S) = \frac{1}{\sigma} \int_\mathbb{R} \int_{G(M)} Dg(x) \cdot S \, dV_{\Sigma_t}(x,S) dt.\]
For each $t$ such that $\Sigma_t$ is smooth with $\mathcal H^{n-1}(\Sigma_t) < \infty$ we have
    \[\int_{G(M)} Dg(x) \cdot S \, dV_{\Sigma_t}(x,S) = \int_{G(X)} Dg(x) \cdot S \, dV_{\Sigma_t}(x,S) = \int_{\Sigma_t} \Div_{\Sigma_t} g(x) \, d\mathcal H^{n-1}(x),\]
and for $x \in \Sigma_t$
   \[\Div_{\Sigma_t} g = \tr(P(w) Dg), \qquad  P(w) := I - \frac{1}{|dw|^2} (\grad w \otimes dw),\]
so we may write
    \[\delta V (g) = \frac{1}{\sigma}\int_\mathbb{R} \int_{\Sigma_t} \tr(P(w) Dg) \, d\mathcal H^{n-1} dt.\]
By the coarea formula, 
    \begin{align*}
        \int_\mathbb{R} \int_{\Sigma_t} \tr(P(w) Dg) \, d\mathcal H^{n-1} dt = \int_{X} \tr(P(w) Dg) |d w| \, dx,
    \end{align*}
and hence
    \[\delta V (g) = \frac{1}{\sigma}\int_{X} \tr(P(w) Dg) |d w|.\]
At points where $|dw|$ is nonzero we have 
    \[\tr(P(w)Dg) = \Div(g) - \frac{1}{|\nabla u|^2} \langle \nabla_{Dg} u, \nabla u\rangle,\]
so we may also write
    \begin{equation}\label{first variation V}
        \delta V (g) = \frac{1}{\sigma}\int_{X} \bigg(\Div(g) - \frac{1}{|\nabla u|^2} \langle \nabla_{Dg} u, \nabla u\rangle\bigg)|d w|.
    \end{equation}

\subsection{Convergence to a stationary varifold}\label{subsec convergence to varifold} Suppose now that we are in the setting of \eqref{local}, and let $u_k$ be a sequence of critical sections as in \eqref{sequence of sections}. Then, possibly after passing to a subsequence, we may assume that the energy measures $\mu_k$ weak*-converge to a Radon measure $\mu$ on $B$. Moreover, since $\|V_k\| \leq \frac{1}{2\sigma}\mu_k$, we may assume that the associated varifolds $V_k$ weak*-converge to a varifold $V$ on $\mathbb{R}^n$. 

In the remainder of this section we study the limiting varifold $V$. In particular, after establishing some lemmas, we prove Theorem~\ref{main critical}.

\begin{lemma}\label{limit monotonicity}
There are constants $\Lambda=\Lambda(n,\Gamma)$ and $r_0 = r_0(n,\Gamma)$ such that for every $p \in \Gamma$ we have
    \begin{equation}\label{limit monotonicity ineq}
        \frac{e^{\Lambda s}}{s^{n-1}}\|V\|(B_s(p)) \leq \frac{e^{\Lambda r}}{r^{n-1}}\|V\|(B_r(p)),
    \end{equation}
for every $0 < s < r <\min\{r_0, \dist(p,\partial B)\}$. The same holds for $p \in X$, for every $0 < s < r < \min\{\dist(p, \Gamma), \dist(p, B)\}$, with $\Lambda = 0$. 
\end{lemma}

\begin{proof}
Consider some fixed $p \in \Gamma$. Let $r_0$ be the constant appearing in Proposition~\ref{energy monotonicity}. Combining that proposition with Proposition~\ref{discrepancy integral} yields \eqref{limit monotonicity ineq}.

If $p \in X$ we apply the same argument, but using the almost-monotonicity formula for interior points proven in \cite{HT}.
\end{proof}

Next we observe that Lemma~\ref{local energy bounds} implies upper and lower bounds for the rescaled energy at points in its support.  

\begin{lemma}\label{limit density}
There are constants $r_1$ and $D$ depending only on $n$, $\Gamma$, $C_0$ and $C_1$ such that 
    \begin{equation}
        D^{-1} \leq \frac{1}{r^{n-1}}\|V\|(B_r(p)) \leq D\bigg(1+\frac{1}{\dist(p,\partial B)^{n-1}}\bigg)
    \end{equation}
for every $p \in \supp \mu$ and $0 < r < \min\{r_1, \dist(p, \partial B)/2\}$. In particular, $\|V\|(\Gamma \cap B) = 0$. 
\end{lemma}
\begin{proof}
The upper and lower bounds for the rescaled mass follow immediately from Lemma~\ref{local energy bounds}. That $\|V\|(\Gamma \cap B) = 0$ follows from a simple covering argument, since $\mathcal H^{n-1}(\Gamma) = 0$. 
\end{proof}

We can now proceed with proving Theorem~\ref{main critical}. 

\begin{proof}[Proof of Theorem~\ref{main critical}]
\emph{Step 1: Equidistribution of energy.} The function $w_k$ is locally Lipschitz, hence differentiable almost everywhere in $X$, and satisfies 
    \[|dw_k| = \sqrt{\frac{W(u_k)}{2}} |\nabla u_k|\]
almost everywhere in $X$. Consequently,
    \begin{align*}
    \bigg|\varepsilon_k \frac{|\nabla u_k|^2}{2} + \frac{W(u_k)}{\varepsilon_k} - 2|d w_k|\bigg| &= \Bigg( \sqrt{\frac{\varepsilon_k}{2}} |\nabla u_k|  - \sqrt{\frac{W(u_k)}{\varepsilon_k}}\Bigg)^2 \leq |\xi_k|.
    \end{align*}
The right-hand side converges to 0 in $L^1_{\loc}(B)$ by Proposition~\ref{discrepancy integral}. It follows that
\begin{equation}\label{nabla u nabla w}
    \bigg|\varepsilon_k \frac{|\nabla u_k|^2}{2} - |d w_k|\bigg| \to 0, \qquad \bigg|\frac{W(u_k)}{\varepsilon_k} - |d w_k|\bigg| \to 0
    \end{equation}
in $L^1_{\loc}(B)$. Furthermore, since for every open $A \subset B$ we have
    \[\|V_k\|(A) = \frac{1}{\sigma}\int_{A\setminus\Gamma} \sqrt{\frac{W(u_k)}{2}} |\nabla u_k| = \frac{1}{2\sigma} \mu_k(A) -  \frac{1}{2\sigma} \int_{A} \Bigg( \sqrt{\frac{\varepsilon_k}{2}} |\nabla u_k|  - \sqrt{\frac{W(u_k)}{\varepsilon_k}}\Bigg)^2,\]
we find that $\|V\|$ and $\mu$ are related by
    \[\|V\| = \frac{1}{2\sigma}\mu.\]
Combining this identity with Lemma~\ref{support on loops}, we obtain 
    \[\Gamma \subset \supp\mu = \supp \|V\|.\]

\emph{Step 2: Density bounds.} Let us define 
    \[\Sigma := \supp \|V\|.\]
By Step~1, $\Sigma = \supp \mu$. In light of Lemma~\ref{limit monotonicity}, for each $p \in \Sigma$ the density 
    \[\Theta^{n-1}(\|V\|, p) := \lim_{r \to 0} \frac{\|V\|(B_r(p))}{ \omega_{n-1} r^{n-1}}, \qquad \omega_{n-1} := |B^{n-1}_1(0)|,\]
exists and is finite. Lemma~\ref{limit density} implies that
    \begin{equation} \label{density} D^{-1} \leq \Theta^{n-1}(\|V\|, p) \leq D\bigg(1+\frac{1}{\dist(p,\partial B)^{n-1}}\bigg)\end{equation}
for $p \in \Sigma$, where the constant $D >0$ depends only on $n$, $\Gamma$, $C_0$ and $C_1$. 

\emph{Step 3: $V$ is stationary.} Let $g \in C^1_0(B, \mathbb{R}^n)$ be such that $g|_{\Gamma}$ is tangent to $\Gamma$. We claim that $\delta V(g) = 0$. Since $\delta V_k(g) \to \delta V(g)$, it suffices to show that $\delta V_k(g) \to 0$. We first recall the identity \eqref{first variation V},
    \[\delta V_k (g) = \frac{1}{\sigma}\int_{X} \bigg(\Div(g) - \frac{1}{|\nabla u_k|^2} \langle \nabla_{Dg} u_k, \nabla u_k\rangle\bigg)|d w_k|.\]
By Lemma~\ref{first variation}, for all $k$ we have 
    \begin{align*}
        \int_X \bigg(\Div(g) - \frac{1}{|\nabla u_k|^2} \langle \nabla_{Dg} u_k, \nabla u_k\rangle\bigg) \varepsilon_k|\nabla u_k|^2 = \int_X \xi_k \Div g,
    \end{align*}
and hence 
    \begin{align*}
        \delta V_k(g) &= \frac{1}{2\sigma} \int_X \xi_k \Div g + \frac{1}{\sigma}\int_X \bigg(\Div(g) - \frac{1}{|\nabla u_k|^2} \langle \nabla_{Dg} u_k, \nabla u_k\rangle\bigg) \bigg(|d w_k| - \frac{\varepsilon_k}{2}|\nabla u_k|^2\bigg).
    \end{align*}
The right-hand side tends to 0 as $k\to\infty$ by \eqref{nabla u nabla w} and Proposition~\ref{discrepancy integral}, so $\delta V(g) = 0$. 

\emph{Step 4: $V$ is integer rectifiable.} We know that $V$ is stationary with respect to vector fields that are compactly supported in $X$. Moreover, $\Theta^{n-1}(\|V\|,p)$ is uniformly positive for $p \in \Sigma$ by Step~2. Therefore, we may apply Allard's rectifiability theorem \cite[5.5 (1)]{Allard} to conclude that $V \mres G(X)$ is a rectifiable $(n-1)$-varifold on $X$. Since $\|V\|(B\cap\Gamma) = 0$, the density upper bound in \eqref{density} ensures that $V$ is a rectifiable varifold on $B$. Finally, $V$ has integer multiplicity $\mathcal H^{n-1}$-almost everywhere in $X$, and hence $\mathcal H^{n-1}$-almost everywhere in $B$, by \cite[Section 5]{HT}. 

\emph{Step 5: Hausdorff convergence of level-sets.} Let us define $\Sigma := \supp \|V\|$. Theorem~1 in \cite{HT} implies that $|u_k| \to 1$ in $C^0_{\loc}(B \setminus \Sigma)$. 

Fix some $b\in(0,1)$. We prove that the closure of 
    \[\{x \in X : |u_k(x)| \leq 1-b\}\]
in $B$, which (by Proposition~\ref{boundary derivatives}) is precisely
    \[\{x \in X : |u_k(x)| \leq 1-b\} \cup \Gamma,\]
converges to $\Sigma$ in the local Hausdorff sense. Let us write 
    \[S_{k,b} := \{x \in X : |u_k(x)| \leq 1-b\} \cup \Gamma\]
and, for each $r > 0$, define 
    \[\Sigma^r := \{x \in B: \dist(x,\Sigma) < r\}, \qquad S^r_{k,b} := \{x \in B : \dist(x,S_{k,b}) < r\}.\]
The claim is then that, for every compact $K$ and positive constant $r$, when $k$ is sufficiently large we have 
    \[\Sigma \cap K \subset S_{k,b}^r \cap K, \qquad \text{and} \qquad S_{k,b} \cap K \subset \Sigma^r \cap K.\]

We first show that, for each $r > 0$, $\Sigma \subset S^r_{k,b}$ for all large $k$. If not let $p \in \Sigma$ be such that 
    \[\limsup_{k \to \infty} \dist(p, S_{k,b}) \geq r.\]
We then have $\dist(p, \Gamma) \geq r$ by Lemma~\ref{support on loops}. Using \cite[Proposition~4.2]{HT} we conclude that
    \[E(u_k, B_s(p)) \to 0\]
for every $s < r$, but this contradicts $p \in \Sigma$.

Next we show that, for each $r > 0$, 
    \[S_{k,b} \cap K \subset \Sigma^r \cap K\]
for all large $k$. If not, let $p_k \in S_{k,b} \cap K$ be a sequence of points such that 
    \[\limsup_{k \to \infty} \dist(p_k, \Sigma) \geq r.\]
Since $\Gamma \subset \Sigma$ it follows that $\dist(p_k, \Gamma) \geq r$. But then we may apply \cite[Proposition~4.2]{HT} to conclude that every accumulation point of $p_k$ is in $\Sigma$. This is a contradiction.
\end{proof}

As we mentioned in the introduction, Theorem~\ref{main critical} (and the interior theory of \cite{HT}) implies a corresponding global statement, for critical sections of the spanning bundle over $\mathbb{R}^n \setminus \Gamma$; see Theorem~\ref{main critical global}. In that theorem, $\supp\|V\|$ is claimed to be compact. This follows easily from Lemma~\ref{limit monotonicity} and Lemma~\ref{limit density}, since $\|V\|$ is bounded.

\subsection{Tangent cones at the boundary}
In the following sections we will address questions concerning the boundary regularity of the limiting varifold which arises in Theorem~\ref{main critical}. Much of our analysis will involve studying rescalings of this varifold at boundary points. Let us conclude this section by laying down some notation and basic results concerning such rescalings. 

Let $V$ be a varifold on an open subset of $\mathbb{R}^n$ and consider a point $p \in \supp\|V\|$. Let $D_{s_i, p}(x):=s_i(x-p)$, where $s_i$ is a sequence of scales $s_i \to \infty$. We refer to any subsequential weak*-limit of the sequence $(D_{s_i, p})_{\#}(V)$ as a varifold tangent to $V$ at $p$. 

\begin{lemma}\label{tangent cone}
In the setting of \eqref{local}, let $u_k$ be a sequence of critical sections as in \eqref{sequence of sections}. Suppose the varifolds $V_k$ associated with $u_k$ weak*-converge to $V$. For every point $p \in \supp \|V\|$, the limit $V$ admits a varifold tangent $C_pV$ at $p$, and the projection of $x$ onto $S^\perp$ vanishes for $C_pV$-almost every $(x, S) \in G(B)$.
\end{lemma}

\begin{proof}
If $p$ is in $\supp\|V\| \setminus \Gamma$ the claim is an easy consequence of the mass bounds \eqref{limit monotonicity ineq} and the monotonicity formula for stationary varifolds \cite[5.1~(1)]{Allard}. 

Suppose then that $p \in \supp\|V\| \cap \Gamma$. In this case the argument is similar, but some technicalities need to be addressed. Fix a sequence of scales $s_i \to \infty$. The varifolds $(D_{s_i, p})_{\#}(V)$ have uniformly bounded mass on compact subsets of $\mathbb{R}^n$ by \eqref{limit monotonicity ineq}. Therefore, after passing to a subsequence, we may assume that $(D_{s_i, p})_{\#}(V)$ weak*-converges to a varifold tangent to $V$ at $p$, which we denote by $C_pV$. Moreover,
    \[\frac{1}{r^{n-1}} \|C_p V\|(B_r(0)) = \lim_{s \to 0} \frac{1}{s^{n-1}}\|V\|(B_{s}(p))\]
for all $r > 0$. 

For each vector field $g \in C^1_0(\mathbb{R}^n \setminus T_p\Gamma, \mathbb{R}^n)$, and all sufficiently large $i$, we have that $T_p\Gamma$ lies outside of the support of $g$. It follows that $(D_{s_i, p})_{\#}(V)$ is stationary with respect to $g$ for all large $i$, and hence $C_p V$ is stationary with respect to $V$. 

Let $\theta$ denote the reflection map across $T_p\Gamma$. By Allard's reflection principle \cite[3.2]{Allard_boundary}, the varifold $C_p V + \theta_\#(C_p V)$ is stationary with respect to each $g \in C^1_0(\mathbb{R}^n, \mathbb{R}^n)$. Since the rescaled mass $r^{1-n}\|C_p V\|(B_{r}(0))$ is constant in $r$, the monotonicity formula for stationary varifolds \cite[5.1 (2)]{Allard} implies that the projection of $x$ onto $S^\perp$ vanishes for $C_pV$-almost every $(x, S) \in G(B)$.
\end{proof}

\begin{remark}
In the proof of Lemma~\ref{tangent cone}, one can also use Allard's boundary version of the monotonicity formula for stationary varifolds (see \cite[3.4~(2)]{Allard_boundary} and \cite{Bourni}), rather than \eqref{limit monotonicity ineq}, to get uniform mass bounds for the sequence $(D_{s_i, p})_{\#}(V)$. Note however that Allard's result assumes $\|V\|(\Gamma \cap B) = 0$, which we proved using \eqref{limit monotonicity ineq}. 
\end{remark}


\section{Boundary behaviour in dimension 3}\label{section dimension 3}

We continue working in the setting of \eqref{local}, with $u_k$ a sequence of critical sections as in \eqref{sequence of sections}. In addition, we assume the ambient dimension is $n =3$, and that the sections $u_k$ have uniformly bounded Morse index in $X$. That is, there is a positive constant $I_0 \in \mathbb{N}$ such that 
    \[\sup_k \ind(u_k, X) \leq I_0.\]

By Theorem~\ref{main critical}, possibly after passing to a subsequence, we may assume the associated varifolds $V_k$ associated with $u_k$ weak*-converge to an integer rectifiable varifold $V$ on $B$. Moreover, we may assume that the energy measures $\mu_k$ weak*-converge to $\mu = \|V\|/2\sigma$. Let $\Sigma := \supp\|V\|$. 

By Corollary~\ref{stable regularity} (which, we recall, was stated as an immediate consequence of \cite{TW} and \cite{Guaraco}), we know that $\Sigma \setminus \Gamma$ is a smooth minimal surface in $X$. Let us remark that, since we are assuming $n=3$, this statement can also be obtained from the level-set estimates of Chodosh--Mantoulidis \cite{Chodosh-Mantoulidis}. 

We are interested in the behaviour of $\Sigma$ near the boundary $\Gamma$. Fix an arbitrary point $p \in \Gamma$. Due to Lemma~\ref{tangent cone}, we know we can extract a varifold tangent to $V$ at $p$. Let $C_p V$ denote such a tangent. By definition, there is a sequence of scales $s_i \to 0$ such that $C_p V$ is the weak*-limit of $(D_{s_i, p})_{\#}(V)$ as $i \to \infty$, where $D_{s_i, p}(x) := s_i^{-1}(x-p)$. We introduce the notation
    \[C_p \Sigma := \supp\|C_p V\|.\]

Let us first demonstrate that $C_p \Sigma$ is smooth away from $T_p \Gamma$. This follows since $C_p \Sigma$ can be obtained as the energy concentration set of a sequence of appropriate rescalings of $u_k$. We define
    \[\tilde B_i := D_{s_i, p}(B), \qquad \tilde \Gamma_i := D_{s_i, p}(\Gamma), \qquad \tilde X_i := \tilde B_i \setminus \tilde \Gamma_i,\]
and define $\tilde L_i \to \tilde X_i$ to be the pullback of $L$ via $D_{s_i, p}^{-1}$. Note that $\tilde L_i$ is simply the nontrivial real line bundle over $\tilde X_i$. We equip $\tilde L_i$ with the bundle metric and connection, still denoted $\langle \cdot, \cdot \rangle$ and $\nabla$, obtained by pulling back those on $L$. With this notation, 
        \[\tilde u_{k,i}(x) := u_{k}(s_i x + p)\]
is a smooth section of $\tilde L_i$ which is critical for the energy $E_{s_i\varepsilon_k}$. Clearly $|\tilde u_{k,i}| \leq C_0$, and by Lemma~\ref{limit monotonicity}, for each ball $B' \subset \mathbb{R}^3$ we have 
    \[E_{s_i \varepsilon_k}(u_{k,i}, B') \leq C\]
for all sufficiently large $i$, where $C = C(n,\Gamma, C_1)$. 

For each $k$ and $i$, let $\tilde V_{k,i}$ denote the varifold associated with $\tilde u_{k,i}$. It is straightforward to check that 
    \[\tilde V_{k,i} = (D_{s_i, p})_{\#}(V).\]
Moreover, for an appropriate subsequence $k_i$, we may assume $\tilde V_{k_i,i}$ weak*-converges to $C_p V$ as $i \to \infty$. (In particular, we require that $\tilde \varepsilon_i := s_i^{-1} \varepsilon_{k_i} \to 0$ as $i \to \infty$.) Since the rescalings $\tilde u_i := u_{k_i, i}$ satisfy
    \[\ind(\tilde u_i, \tilde X_i) = \ind(u_{k_i}, X) \leq I_0,\]
by Corollary~\ref{stable regularity}, the set $C_p \Sigma \setminus T_p \Gamma$ is a smooth minimal surface. Moreover, by Lemma~\ref{tangent cone}, $C_p \Sigma$ is a cone. Note we also know that $C_p V$ has integer multiplicity by Theorem~\ref{main critical} (this can alternatively be deduced from the compactness theorem for integer rectifiable varifolds \cite[6.4]{Allard}). 

Our goal now is to characterise $C_p \Sigma$ as a union of halfplanes. 

\begin{lemma}
The set $C_p\Sigma$ consists of a collection of halfplanes $P_1, \dots, P_N$ which meet along $T_p \Gamma$. In particular, there are positive integers $m_1, \dots, m_N$ such that $C_p V = \sum_{j=1}^N m_j V_{P_j}$.
\end{lemma}

\begin{proof}
Let $p_{\pm}$ be the points where $T_p\Gamma$ intersects $S^2 = \partial B_1(0)$. Since $C_p \Sigma$ is a smooth minimal cone, its intersection with $S^2 \setminus \{p_{\pm}\}$ is a collection of geodesics $\alpha_1, \dots, \alpha_N$. Since $C_p \Sigma$ is closed and hence properly embedded in $\mathbb{R}^3 \setminus T_p\Gamma$, we either have that the $\alpha_1, \dots, \alpha_N$ are all half-circles with endpoints at $p_{\pm}$, or else $N =1$ and $\alpha_1$ is a great circle. 

Suppose $N=1$ and $\alpha_1$ is a great circle. We then have that $|\tilde u_i| \to 1$ locally uniformly in $S^2 \setminus (\alpha_1 \cup \{p_{\pm}\})$, by Theorem~\ref{main critical}. But $\tilde u_i$ must vanish on any loop $\gamma$ in $\mathbb{R}^3 \setminus T_p \Gamma$ which satisfies $\link_{\mathbb{Z}}(\gamma, T_p \Gamma) = 1$. We can find such a loop in $S^2 \setminus (\alpha_1 \cup \{p_{\pm}\})$, so we have reached a contradiction. 

We conclude that the intersection of $C_p \Sigma$ with $S^2 \setminus \{p_{\pm}\}$ consists of half-circles $\alpha_1, \dots, \alpha_N$ with endpoints at $p_{\pm}$. Equivalently, $C_p\Sigma$ is the set of halfplanes $P_1, \dots, P_N$ generated by $\alpha_1, \dots, \alpha_N$. Given that $C_p V$ has integer multiplicity, by the constancy theorem (see e.g. \cite[4.1]{SimonGMT}), there are positive integers $m_1, \dots, m_N$ such that $C_p V = \sum_{j=1}^N m_j V_{P_j}$. 
\end{proof}

We conclude this section with the proof of Theorem~\ref{stable boundary cone}. It only remains to establish that $\sum_{j=1}^N m_j$ is odd. 

\begin{proof}[Proof of Theorem~\ref{stable boundary cone}]
We construct a loop $\gamma$ in $\mathbb{R}^3 \setminus T_p\Gamma$ such that $\link_{\mathbb{Z}}(\gamma, T_p \Gamma) = \pm 1$ and $\tilde u_i$ has $m_1 + \dots + m_N$ isolated zeroes on $\gamma$ for all large $i$. Since $\tilde L_i$ is nontrivial, it then follows that $m_1 + \dots + m_N$ is odd. 

For each $j = 1, \dots, N$, let $p_j$ denote the point where $P_j$ meets the equator $S^2 \cap N_p \Gamma$. Let $\nu_j$ denote a unit normal to $P_j$. For each $j$ we define a truncated cylinder 
    \[Z_j := \{x + s\nu_j : (x-p_j) \cdot \nu_j = 0, \; |x-p_j|< \delta, \; |s| < \delta\}.\]
By choosing $\delta$ to be sufficiently small we can arrange that $Z_j \cap C_p \Sigma = Z_j \cap P_j$. We now let $p_{j,l}$, $l = 1, \dots, I_0 + 1$, be a collection of points in $Z_j \cap P_j$, and define for each $p_{j,l}$ a thinner truncated cylinder inside $Z_j$, as follows:
    \[Z_{j,l} := \{x + s\nu_j : (x-p_{j,l}) \cdot \nu_j = 0, \; |x-p_{j,l}|< 10^{-2}(I_0 + 1)^{-1}\delta, \; |s| < \delta\}.\]
We assume the points $p_{j,l}$ are chosen so that the $Z_{j,l}$ are all mutually disjoint. 

Since $\ind(\tilde u_i, \tilde X_i) \leq I_0$, for every $j$, we know that $\tilde u_i$ is stable in at least one of the cylinders $Z_{j,l}$. Possibly after passing to a subsequence in $i$, and relabeling indices, we may assume $\tilde u_i$ is stable in $Z_j' := Z_{j,1}$ for every $j$. Recall that the associated varifolds $\tilde V_i$ weak*-converge to $m_j V_{P_j}$ on $Z_j$, and hence on $Z_j'$. Let 
    \[\hat Z_j := \{x + s\nu_j : (x-p_{j,1}) \cdot \nu_j = 0, \; |x-p_{j,1}|< 10^{-3}(I_0 + 1)^{-1}\delta, \; |s| < \delta\}.\]

We analyse the convergence of $\{\tilde u_i = 0\}$ in each of the cylinders $Z_j'$. The argument is the same in each cylinder, so let us drop the index $j$, understanding that $P = P_j$, $m = m_j$, etc. Given that $\tilde u_i$ is stable in $Z'$, the curvature estimate of Chodosh--Mantoulidis \cite[Theorem~1.3]{Chodosh-Mantoulidis} implies that, for all sufficiently large $i$, the set $\{\tilde u_i = 0\} \cap \hat Z$ is a union of $\tilde m_{i}$ smooth graphs over $P$, each of which converges in $C^{1,\alpha}$ to $P \cap \hat Z$ as $i \to \infty$. 

Standard arguments show that $\tilde m_{i} = m$ for all large $i$. We only sketch the details. We define
    \[T_i := \{x \in \hat Z : \dist(x, \{\tilde u_i = 0\})\} \leq \Lambda \varepsilon, \]
where $\Lambda$ is a large constant to be chosen later. We recall that on $\mathbb{R}^3$ the only entire solutions of the scalar Allen--Cahn equation which are bounded, stable, and have quadratic area growth, are the 1-dimensional heteroclinic solutions and the constants $\pm 1$.\footnote{This is proven using the log cutoff trick. The argument is in \cite{FMV}. See also \cite{Ambrosio--Cabre} for related classification results.} Using this fact, together with a blow-up argument, it can be shown that if $\Lambda$ is sufficiently large then
    \[\limsup_{i \to \infty} |E_{\tilde \varepsilon_i}(\tilde u_i, T_i) - 2\sigma \tilde m_{i} \mathcal H^2(P \cap \hat Z)| \leq Ce^{-\Lambda/C},\]
where $C >0$ is a universal constant and, we recall, $2\sigma$ is the energy of the 1-dimensional heteroclinic solution. Moreover,
    \[\dist(\hat Z \setminus T_i, S_i) \geq \frac{\Lambda}{2} \tilde \varepsilon_i, \qquad \text{where} \qquad S_i := \{|\tilde u_i|\leq 9/10\}.\]
Standard interior estimates for linear elliptic equations show that for $x \in \hat Z \setminus T_i$ we have
    \[\tilde \varepsilon_i \frac{|\nabla \tilde u_i(x)|^2}{2} + \frac{W(\tilde u_i(x))}{\tilde \varepsilon_i} \leq C\tilde \varepsilon_i^{-1} e^{-\tilde \varepsilon_i^{-1}\dist(x, S_i)/C}.\]
Therefore, we may use the coarea formula to estimate 
    \begin{align*}
        E_{\tilde \varepsilon_i}(\tilde u_i, \hat Z \setminus T_i)  &\leq C r^2 \tilde \varepsilon_i^{-1} \int_{\Lambda \tilde \varepsilon_i / 2}^\infty e^{-\tilde \varepsilon_i^{-1}t/C} dt \leq C r^2 e^{- \Lambda/C}.
    \end{align*}
Combining all of this, we see that by choosing $\Lambda$ sufficiently large we can ensure that $E_{\tilde \varepsilon_i}(\tilde u_i, \hat Z)$ is as close as we like to $2\sigma \tilde m_{i} \mathcal H^2(P \cap \hat Z)$ for all large $i$. On the other hand,
    \[E_{\tilde \varepsilon_i}(\tilde u_i, \hat Z) \to 2\sigma m \omega_2 r^2\]
as $i \to \infty$. It follows that $\tilde m_{i} = m$ for all large $i$. 

To recap, we have shown that $\{\tilde u_i = 0\} \cap \hat Z_j$ has exactly $m_j$ graphical components whenever $i$ is sufficiently large (we now resume using the index $j$ to distinguish the planes $P_j$). We may deform the loop $\gamma := S^2 \cap N_p \Gamma$ slightly so that we still have $\link_{\mathbb{Z}}(\gamma, T_p \Gamma) = \pm 1$, but now 
    \[\gamma \cap Z_{j} \subset \hat Z_j\]
for each $j$. Moreover, we may assume $\gamma$ is monotone inside $\hat Z_j$, so that it intersects $\{\tilde u_i = 0\}$ exactly $m_j$ times inside $\hat Z_j$. Since $|\tilde u_i| \to 1$ locally uniformly away from $P_1 \cup \dots \cup P_N$, we may assume that 
    \[\{\tilde u_i = 0\} \cap \gamma \subset \hat Z_1 \cup \dots \cup \hat Z_N.\]
It follows that $\tilde u_i$ has $m_1 + \dots + m_N$ isolated zeroes on $\gamma$, and hence $m_1 + \dots + m_N$ is odd. 
\end{proof}


\section{Minimizing sections and Plateau's problem}\label{minimizers}

Let $\Gamma$ denote a compact codimension-2 submanifold of $\mathbb{R}^n$, where $n \geq 3$. In this section we write $L$ for the spanning bundle over $X = \mathbb{R}^n \setminus \Gamma$. A section $u \in W^{1,2}_{\loc}(X, L)$ is said to be minimizing for $E_\varepsilon$ if 
    \[E_\varepsilon(u) \leq E_{\varepsilon}(v)\]
for every section $v \in W^{1,2}_{\loc}(X,L)$. We are concerned with the global behaviour of minimizing sections in the limit as $\varepsilon \to 0$. Our ultimate goal is to prove Theorem~\ref{Plateau theorem}, which asserts that minimizers concentrate energy on a solution of Plateau's problem (at least when $n = 3$, or when $4 \leq n \leq 7$ and $\Gamma$ lies in a strictly convex hypersurface). 

We first prove that at least one minimizer of $E_\varepsilon$ exists for each positive $\varepsilon$. The proof is a standard application of compactness theorems for Sobolev functions. 

\begin{lemma}\label{existence minimizers}
For every $\varepsilon > 0$ there exists a smooth section $u$ of $L$ which minimizes $E_\varepsilon$.
\end{lemma}

\begin{proof}
Fix $\varepsilon > 0$ and let $\bar E$ be the infimum of $E_\varepsilon(v)$, taken over all sections $v \in W^{1,2}_{\loc}(X,L)$. We have $\bar E < \infty$. Let $u_k$ be a sequence of $W^{1,2}_{\loc}$-sections such that $E_{\varepsilon}(u_k) \to \bar E$ as $k \to \infty$. After passing to a subsequence, we may assume $E_{\varepsilon}(u_k) \leq 1 + \bar E$. In any ball $B \subset X$ we have
    \[\bigg(\int_B |u|^2\bigg)^2 \leq \mathcal{H}^n(B) \int_B |u|^4 \leq  \mathcal{H}^n(B) \bigg(\mathcal{H}^n(B\cap\{|u|^2\leq 2\}) + 16\int_B W(u)\bigg),\]
so $\|u_k\|_{W^{1,2}}(B)$ is bounded independently of $k$. Choosing a countable covering of $X$ by open balls and appealing to a diagonal argument, we find that there is a section 
    \[u \in W^{1,2}_{\loc}(X, L)\]
such that $u_k \to u$ weakly in $W^{1,2}(B)$ for every open ball $B \subset X$. By the Sobolev embedding theorem, we may also assume that $u_k \to u$ in $L^q(B)$ for every $B$, where $q = \tfrac{1}{2} - \tfrac{1}{n}$. Then $u_k \to u$ pointwise almost everywhere (possibly after passing to another subsequence), and hence by Fatou's lemma we have 
    \[\int_X \frac{W(u)}{\varepsilon} \leq \liminf_{k\to\infty}\int_X \frac{W(u_k)}{\varepsilon}.\]
Since the Dirichlet integral is lower semicontinuous with respect to weak convergence in $W^{1,2}$, we have 
    \[E_{\varepsilon}(u) \leq \liminf_{k\to\infty}E_\varepsilon(u_k) = \bar E,\]
which is to say that $u$ minimizes $E_{\varepsilon}$. Finally, since $u$ is in particular critical for $E_\varepsilon$, it is smooth. 
\end{proof}

We would like to apply Theorem~\ref{main critical global} to sequences of minimizers with $\varepsilon \to 0$. To do so we need to show that minimizers satisfy uniform length and energy bounds. 

\begin{lemma}\label{length}
Fix $\varepsilon > 0$ and suppose $u$ is a minimizing section for $E_\varepsilon$. We then have 
        \[\sup_X |u| \leq 1.\]
\end{lemma}
\begin{proof}
Let $e := u/|u|$ on the set $\{|u| > 0\}$. We then define a new section $\tilde u$ such that 
    \[\tilde u = \min\{|u|,1\} e\]
on the set $\{|u| > 0\}$, and $\tilde u = 0$ elsewhere. It is not difficult to check that $\tilde u$ is locally Lipschitz continuous and, unless $|u| \leq 1$ holds everywhere, $E_{\varepsilon}(\tilde u) < E_{\varepsilon}(u)$. Given that $u$ is minimizing this is impossible, so $|u| \leq 1$.  
\end{proof}

We now construct competitors which can be used to bound the energy of a minimizer independently of $\varepsilon$. The construction becomes particularly simple when $\Gamma$ is the boundary of a smooth hypersurface, but this is not always the case for $n \geq 4$. As a result, some care is needed in a neighbourhood of $\Gamma$. 

\begin{lemma}\label{global energy}
Consider a constant $\varepsilon_0 > 0$. For each $\varepsilon \in (0,\varepsilon_0)$ there is a locally Lipschitz section $v$ of $L$ such that $E_\varepsilon(v) \leq C$, where $C = C(n,\Gamma, \varepsilon_0)$ is a constant. 
\end{lemma}
\begin{proof}
Fix a smooth section $w \in C^\infty(X,L)$. Although the nodal set of $w$ may not be a smooth hypersurface, it can be perturbed to a section which does have this property, as follows.

First we show that $L$ is finitely generated. That is, there is a finite collection of smooth sections $v_i$, $1 \leq i \leq I$, such that, for every $x \in X$, the vectors $v_i(x)$ span the fiber over $x$. This follows easily from the fact that $X$ can be covered by finitely many open sets $U_i$, $1 \leq i \leq I$, over which $L$ trivialises. Since $\Gamma$ is compact, we can take $U_1$ to be the exterior of a large ball containing $\Gamma$. Since we are assuming $n \geq 3$, $U_1$ is simply connected, and hence $L$ is trivial over $U_1$. We then consider a covering of $\Gamma$ by finitely many open $(n-2)$-dimensional balls. A normal neighbourhood of each of these is diffeomorphic to $B^2 \times B^{n-2}$, and can be covered by two simply connected sets by removing thin wedges from the $B^2$ factor. This gives sets $U_i$, $2 \leq i \leq I'$, which cover a tubular neighbourhood of $\Gamma$. The remainder of $X$ can then be covered by finitely many balls, which are the sets $U_i$, $I' \leq i \leq I$. For each $U_i$, we let $v_i$ be a unit-length section in $U_i$, but then multiply it by a cutoff function and extend so that it vanishes outside of $U_i$. This can be done so that at least one of the sections $v_i$ is nonzero at each $x \in X$.

Now let $F: X \times \mathbb{R}^I \to L$ be the smooth map defined by
    \[F(x,s) := w(x) + \sum_{i=1}^I s_i v_i(x).\]
Since at least one of the vectors $v_i$ is always nonzero, $F$ is transverse to the submanifold of $L$ traced out by its zero section. Therefore, by the Thom transversality theorem, for almost every $s$, the map $x \mapsto F(x,s)$ is transverse to the zero section of $L$. In particular, there exists a $\tilde s$ such that the zero-set of
    \[\tilde w(x) := F(x, \tilde s)\]
is a smooth hypersurface in $X$. Moreover, we can assume $\tilde s_1>0$, so that $\tilde w$ is nonzero outisde of a large compact subset containing $\Gamma$, and hence $\tilde w^{-1}(0)$ is a precompact subset of $\mathbb{R}^n$. 

We know $\tilde w^{-1}(0)$ is smooth, but it may not have finite $\mathcal H^{n-1}$-measure. We can rectify this as follows. Let $\delta >0$ be a small constant such that $\{\rho < \delta\}$ is a tubular neighbourhood of $\Gamma$ whose boundary meets $\tilde w^{-1}(0)$ transversally (recall the notation $\rho(x) = \dist(x, \Gamma)$). We define $\Sigma \subset X$ so that 
    \[\Sigma \cap \{\rho \geq \delta\} = \tilde w^{-1}(0) \cap \{\rho \geq \delta\}\]
and, writing $\pi$ for the nearest point projection to $\Gamma$,
    \[\Sigma \cap \{0 < \rho < \delta\} = \{sx + (1-s)\pi(x) : x \in \tilde w^{-1}(0) \cap \{\rho = \delta\}, \; s \in (0,1)\}.\]
Then $\Sigma$ is a locally Lipschitz hypersurface with $\mathcal H^{n-1}(\Sigma) < \infty$. Moreover, there is a unit section $e$ of $L$ which is defined everywhere in $X \setminus \Sigma$. Indeed, for $x \in (X \setminus \Sigma) \cap \{\rho \geq \delta\}$ we can take $e(x) = \tilde w(x)/|\tilde w(x)|$. We then extend $e$ to $(X\setminus\Sigma)\cap\{0<\rho < \delta\}$ by parallel transport (along e.g. normal segments emanating from $\Gamma$). It follows that $e$ is smooth and parallel in $X \setminus \Sigma$.

We define
    \[v(x) := \begin{cases}
                \psi(\dist(x, \Sigma)/\varepsilon) e(x) & \text{for } x \in X\setminus\Sigma,\\
                0 & \text{for } x \in \Sigma,
            \end{cases}
            \]
where $\psi:\mathbb{R}\to\mathbb{R}$ coincides with a 1-dimensional heteroclinic solution on a large set, in the following manner. We let
    \[\psi(z) := (1 - \xi(z/100))\tanh(z/\sqrt{2}) + \xi(z/100),\]
where $\xi : \mathbb{R} \to \mathbb{R}$ satisfies $0 \leq \xi \leq 1$, $\xi \equiv 0$ in $[-1,1]$, $\xi \equiv 1$ in $\mathbb{R}\setminus[-2,2]$ and $|\xi'| \leq 2$. Note that $\dist(\cdot, \Sigma)$ is a locally Lipschitz function on $\mathbb{R}^n$, so $v$ is locally Lipschitz. 

The $\mathcal H^{n-1}$-measure of $\{\dist(\cdot, \Sigma) = t\}$ varies continuously in $t$, and approaches $2 \mathcal H^{n-1}(\Sigma)$ as $t \to 0$. Therefore, there is a constant $\kappa > 0$ depending only on $\Sigma$ such that $\{\dist(\cdot, \Sigma) = t\} \leq 1+2 \mathcal H^{n-1}(\Sigma)$ for $t \leq \kappa$. A straightforward application of the coarea formula shows that when $\varepsilon \leq \kappa/200$ we have
    \[E_\varepsilon(v) \leq C(1+2\mathcal H^{n-1}(\Sigma)),\]
where $C$ is a universal constant. This completes the proof of the lemma for $\varepsilon \leq \kappa/200$. 

If instead $\kappa/200 < \varepsilon < \varepsilon_0$, we can simply choose $v$ so that it vanishes inside a large ball containing $\Gamma$, has unit length outside a somewhat larger ball, and satisfies $|\nabla v| \leq 2$. We then have 
    \[E_\varepsilon(v) \leq C(\varepsilon + \varepsilon^{-1}) \leq C(\varepsilon_0 + 200\kappa^{-1}),\]
where $C$ depends only on $n$ and $\Gamma$. With this the claim is proven.
\end{proof}

\begin{remark}
Throughout this section of the paper we always assume $n \geq 3$. We remark that the proof of Lemma~\ref{global energy} breaks down when $n = 2$. For example, the nontrivial line bundle over $\mathbb{R}^2 \setminus \{0\}$ fails to be finitely generated and, generically, the zero set of a smooth section of this bundle is a family of unbounded curves. 
\end{remark}

We may now apply Theorem~\ref{main critical global} and Corollary~\ref{stable regularity} to establish the following statement. 

\begin{proposition}\label{minimizer interior reg}
Fix a sequence $\varepsilon_k \to 0$, and let $u_k$ denote a sequence of sections of the spanning bundle over $X$ which minimize $E_{\varepsilon_k}$. After passing to a subsequence, the varifolds $V_k$ associated with $u_k$ weak*-converge to some limit, denoted $V$. Let $\Sigma := \supp\|V\|$. We then have that $\Sigma \setminus \Gamma$ is a smooth minimal hypersurface outside a set of Hausdorff-dimension at most $n-8$, and $V$ is the unit-multiplicity varifold induced by $\Sigma$.
\end{proposition}

\begin{proof}
Given Lemma~\ref{length} and Lemma~\ref{global energy}, Theorem~\ref{main critical global} implies that, after passing to a subsequence, $V_k$ weak*-converges to an integer rectifiable limit $V$. Since each $u_k$ is, in particular, stable, $\Sigma \setminus \Gamma$ is a smooth minimal hypersurface outside a set of Hausdorff-dimension at most $n-8$, by Corollary~\ref{stable regularity}. It remains to show that $V$ is the unit-multiplicity varifold induced by $\Sigma$. Since $u_k$ is minimizing, a straightforward cut-and-paste argument shows that points of multiplicity at least two do not arise---see the proof of Theorem~2 on p.~78 of \cite{HT}.
\end{proof}

Next we address the regularity of $\Sigma$ at the boundary $\Gamma$. 

\begin{proposition}\label{minimizer bdy reg}
Let $u_k$, $V$ and $\Sigma$ be as in Proposition~\ref{minimizer interior reg}. Suppose in addition that $n =3$, or else $n \geq 4$ and $\Gamma$ lies in a strictly convex hypersurface. Then there is a tubular neighbourhood of $\Gamma$ in which $\Sigma$ is a smooth hypersurface with boundary $\partial \Sigma = \Gamma$. 
\end{proposition}

\begin{proof}
In case $\Gamma$ lies in a strictly convex hypersurface, by the convex hull property for stationary varifolds \cite[Theorem~6.2]{SimonGMT}, we can simply apply \cite[5.2]{Allard_boundary} to conclude that $V$ has density $1/2$ at each point in $\Gamma$. The claim then follows immediately from Allard's boundary regularity theorem \cite{Allard_boundary}. 

Let us turn to the case $n=3$. Let $C_p V$ be a varifold tangent to $V$ at $p \in \Gamma$. As a consequence of Theorem~\ref{stable boundary cone}, we know that $C_p V$ is induced by halfplanes $P_1, \dots, P_N$ with multiplicities $m_1, \dots, m_N$ respectively, where $m_1 + \dots + m_N$ is odd. 

As in Section~\ref{section dimension 3}, we can find appropriate rescalings $\tilde u_i := u_{k_i, i}$ of $u_k$ whose associated varifolds converge to $C_p V$. We recall our notation---$\tilde u_i$ is a section of $\tilde L_i$ and is minimizing for $E_{\tilde \varepsilon_i}$ (where $\tilde \varepsilon_i \to 0$). By the same cut-and-paste argument referred to above, we can show that each halfplane $P_j$ occurs with multiplicity $m_j = 1$. Consequently, $N$ is odd. 

Suppose now, with the aim of deriving a contradiction, that $N \geq 3$. Then there is a pair of halfplanes in $\supp\|C_p V\|$ which meet at an angle not exceeding $2\pi/3$. Suppose without loss of generality these are the planes $P_1$ and $P_2$. Choose a large ball $B$ in $\mathbb{R}^3 \setminus T_p \Gamma$ such that these two planes divide $B$ into three pieces. If $B$ is large enough, we can cut a hole in each of $P_1 \cap B_R$ and $P_2 \cap B$ and attach them with a neck, in such a way that the resulting surface $S \subset B$ is smooth and has strictly less area then $(P_1 \cup P_2)\cap B$.  

Let us fix a unit section of $\tilde L_i$ over $B$, so that we may treat $\tilde u_i$ as a function on $B$. By a standard construction, we can add to $\tilde u_i$ a function $w_i \in W^{1,2}_0(B)$ so that $E_{\tilde \varepsilon_i}(\tilde u_i + w_i)$ is as close as we like to $2\sigma \mathcal H^2(S)$. Let us sketch the construction. We first fix a smaller ball $B_r \subset B$ so that $S$ and $P_1 \cup P_2$ agree in $B \setminus B_r$ and $\partial B_r$ intersects $S$ transversally. Let $\eta_i$ be a smooth function on $B$ such that $0 \leq \eta_i \leq 1$, $\eta_i = 1$ in $B_r$, $\eta_i = 0$ in $B\setminus B_{r + \delta \tilde \varepsilon_i}$, and $|\nabla \eta_i| \leq 2\delta^{-1}\varepsilon^{-1}$. We also define 
    \[v_i(x) := \begin{cases} \tanh\bigg(\frac{\dist(x,S)}{\tilde \varepsilon_i\sqrt{2}}\bigg) e(x) & \text{for } x \in B \setminus S,\\ 0 & \text{for } x \in S, \end{cases}.\]
By splitting up $B$ as 
    \[B = (B \setminus B_{r+\delta\tilde\varepsilon_i}) \cup (B_{r+\delta\tilde\varepsilon_i} \setminus B_r) \cup B_r\]
and using the coarea formula, we find that for $w_i := \eta_i(v_i - \tilde u_i)$ we have 
    \[|E_{\tilde \varepsilon_i}(\tilde u_i + w_i, B) - 2\sigma \mathcal H^2(S)| \leq C\delta\]
for all sufficiently large $i$, where $C > 0$ is a constant depending only on $n$ and $S$.  

Recall that 
    \[E_{\tilde \varepsilon_i}(\tilde u_i, B) \to 2\sigma \mathcal H^2((P_1\cup P_2) \cap B )\]
as $i \to \infty$. Therefore, since 
    \[\mathcal H^2(S) < \mathcal H^2((P_1 \cup P_2) \cap B),\]
by choosing $\delta$ small enough in the above construction we can ensure that the inequality
    \[E_{\tilde \varepsilon_i}(\tilde u_i + w_i, B) < E_{\tilde \varepsilon_i}(\tilde u_i, B)\]
holds for all large $i$. But this means $\tilde u_i$ is not a minimizer, so we have reached a contradiction. We thus have $N = 1$, and hence $C_p V$ is a halfplane with multiplicity one. It follows that $V$ has density $1/2$ at $p$, so we can apply \cite{Allard_boundary} to conclude that $\Sigma$ is a smooth surface with boundary in a neighbourhood of $p$. 
\end{proof}

\begin{remark}
It would be interesting to know whether the nodal set of a minimizing/stable section is in fact a hypersurface with boundary, at least if $n =3$ and $\varepsilon$ is sufficiently small. This might be proven by developing a regularity theory for nodal sets at boundary points analogous to the corresponding interior theory---see e.g. \cite{CC, Wang, Wang-Wei_a, Chodosh-Mantoulidis}. 
\end{remark}

Theorem~\ref{Plateau theorem} follows by combining Proposition~\ref{minimizer interior reg}, Proposition~\ref{minimizer bdy reg} and the following lemma. 

\begin{lemma}\label{proof of plateau}
Let $u_k$, $V$ and $\Sigma$ be as in Proposition~\ref{minimizer interior reg}. If $\Sigma'$ is a smooth hypersurface with boundary such that $\partial \Sigma' = \Gamma$, then
    \[\mathcal H^{n-1}(\Sigma)\leq \mathcal H^{n-1}(\Sigma')\]
\end{lemma}
\begin{proof}
Suppose, with the aim of deriving a contradiction, that $\Sigma'$ is a smooth hypersurface with boundary such that $\partial \Sigma' = \Gamma$ and 
    \[\mathcal H^{n-1}(\Sigma') < \mathcal H^{n-1}(\Sigma).\]
Let us assume that any connected components of $\Sigma'$ which are disjoint from $\partial \Sigma'$ have been discarded. A straightforward topological argument shows that the only real line bundle over $X \setminus \Sigma'$ is the trivial bundle. In particular, the restriction of $L$ to $X \setminus \Sigma'$ is trivial, and we may take $e : X \setminus \Sigma' \to L$ be a smooth unit section. 

Although we are not assuming $\Sigma'$ is compact, let us suppose for now that this is the case. We define a new section $u_k' : X \to L$, whose energy approximates $2\sigma \mathcal H^{n-1}(\Sigma')$, as follows. We set
    \[u_k'(x) = \begin{cases} \psi(\dist(x,\Sigma')/\varepsilon_k) e(x) & \text{for } x \in X \setminus \Sigma'\\ 0 & \text{for } x \in \Sigma',\end{cases}\]
where $\psi: \mathbb{R} \to \mathbb{R}$ is of the form
    \[\psi(z) := (1 - \xi(z/\Lambda))\tanh(z/\sqrt{2}) + \xi(z/\Lambda)\]
and $\xi : \mathbb{R} \to \mathbb{R}$ satisfies $0 \leq \xi \leq 1$, $\xi \equiv 0$ in $[-1,1]$, $\xi \equiv 1$ in $\mathbb{R}\setminus[-2,2]$ and $|\xi'| \leq 2$. The result, easily obtained from the coarea formula, is that for every ball $B \subset \mathbb{R}^n$ we have
    \[\limsup_{k \to \infty} |E_{\varepsilon_k}(u_k', B) -2\sigma \mathcal H^{n-1}(\Sigma' \cap B)| \leq C(1+\mathcal H^{n-1}(\Sigma' \cap B))e^{-\Lambda/C},\]
where $C > 0$ is a constant which may depend on $n$, $\Gamma$ and $B$. In particular, by choosing $B$ large enough so that it contains $\Sigma$ and $\Sigma'$, and then choosing $\Lambda$ to be sufficiently large depending on $B$, we can ensure that
    \[E_{\varepsilon_k}(u_k')  < E_{\varepsilon_k}(u_k)\]
for all large $k$. This shows that $u_k$ is not a minimizer when $k$ is large, which is a contradiction.

When $\Sigma'$ is noncompact we proceed in the same way, but modify $u_k'$ so that it has unit length outside of a large ball $B$. This can be done in such a way that $E_{\varepsilon_k}(u_k', B)$ is as close as we like to $2\sigma\mathcal H^{n-1}(\Sigma')$ for large $k$. Since $u_k'$ has no energy in the complement of $B$, this again contradicts $u_k$ being a minimizer. We omit the details concerning the modification of $u_k'$ (the construction is straightforward, and very similar to the construction of $w$ in the proof of Proposition~\ref{minimizer bdy reg}). 
\end{proof}


\section{Minimizing area in a homology class} 

Let $(M,g)$ be a compact Riemannian manifold of dimension $n$. A natural problem, sometimes referred to as the homological Plateau problem, is to prove the existence of an area-minimizer in every nonzero class $[\Sigma] \in H_{n-1}(M, \mathbb{Z}_2)$. The existence of such a minimizer follows from Fleming's theory of flat chains mod~2 \cite{Fleming} (see also \cite{Federer-Fleming}), and is one of the landmark achievements of geometric measure theory. 

In this section we discuss alternative approaches to the homological Plateau problem which employ phase transition models such as the Allen--Cahn energy. For a nonzero class $[\Sigma] \in H_{n-1}(M,\mathbb{Z}_2)$, the element of $H^1(M, \mathbb{Z}_2)$ Poincar\'{e}-dual to $[\Sigma]$ gives rise to a nontrivial real line bundle $L \to M$ (see Section~\ref{topology}). The zero-set of a generic smooth section of $L$ is a smooth cycle in $[\Sigma]$. We may equip $L$ with a metric and a flat metric connection, and for each section $u\in W^{1,2}(M,L)$, define
    \[E_\varepsilon(u) := \int_M \varepsilon \frac{|\nabla u|^2}{2} + \frac{(1-|u|^2)^2}{4\varepsilon}\,\dvol_g.\]
For each $\varepsilon > 0$ there is a smooth section $u_\varepsilon$ of $L$ which minimizes $E_\varepsilon$ (this can be proven exactly as in Lemma~\ref{existence minimizers}). Baldo and Orlandi \cite{Baldo--Orlandi} used $\Gamma$-convergence techniques to show that, for a sequence $\varepsilon_k \to 0$, the sections $u_k := u_{\varepsilon_k}$ concentrate energy on an area-minimzing cycle in $[\Sigma]$. In particular, this gave a new proof of the existence of such an area-minimizer. We think it worth remarking that, at least if $2 \leq n \leq 7$, the existence of a smooth area-minimizing representative of $[\Sigma]$ can now be proven entirely using level-set estimates for $u_k$. We do not claim any originality here---the argument proceeds by combining results of other authors in a straightforward manner. 

First, the energy of $u_k$ is bounded from above by a constant depending only on $(M,g)$ and the class $[\Sigma]$. This is established by constructing an appropriate competitor for each $\varepsilon_k$. The construction is very similar to that in Lemma~\ref{global energy}, but now we arrange that the competitors concentrate energy around some fixed smooth hypersurface in the class $[\Sigma]$. 
It follows from \cite{Wang} and \cite{Wang-Wei_a} (see Section~5 of \cite{Chodosh-Mantoulidis} for the argument) that the nodal sets $u_k^{-1}(0)$ subconverge in the graphical $C^{2,\alpha}$-sense, and with multiplicity one, as $k \to \infty$; this is the step in which we require that $n \leq 7$. Since the hypersurfaces $u_k^{-1}(0)$ lie in $[\Sigma]$ for all large $k$, their limit does as well. So let us denote the limit by $\Sigma$. It remains to show that $\Sigma$ is area-minimizing. This follows since, if there were some other smooth representative of $[\Sigma]$ with less area than $\Sigma$, then for $k$ sufficiently large we could construct a section with less energy than $u_k$ (as in Lemma~\ref{proof of plateau}).

\bibliographystyle{alpha}
\bibliography{references}

\end{document}